\documentclass[11pt]{article}
\usepackage{graphicx,amsmath,amsfonts,amssymb,bbm,amsthm}
\usepackage{psfrag}
\usepackage{epsf}

\makeatletter\@addtoreset {equation}{section}\makeatother

\theoremstyle{plain}
\newtheorem{theorem}{Theorem}[section]
\newtheorem{lemma}[theorem]{Lemma}

\newtheorem{proposition}[theorem]{Proposition}
\newtheorem{corollary}[theorem]{Corollary}
\theoremstyle{remark}
\newtheorem{remark}[theorem]{Remark}

{\hspace*{\fill}$\rule{.3\baselineskip}{.35\baselineskip}$\end{trivlist}}

\newcommand{\R}{\mathbb{R}}
\newcommand{\C}{\mathbb{C}}
\newcommand{\Z}{\mathbb{Z}}

\newcommand{\I}{\mathcal{I}}

\renewcommand{\phi}{\varphi}

\newcommand{\dd}{\,\mathrm{d}}
\newcommand{\half}{{\textstyle\frac12}}

\begin{document}

\title{\bf Orbital stability in the cubic defocusing NLS equation: I. Cnoidal periodic waves}

\author{Thierry Gallay$^{1}$ and Dmitry Pelinovsky$^{2}$ \\
{\small $^{1}$ Institut Fourier, Universit\'e de Grenoble 1,
38402 Saint-Martin-d'H\`eres, France} \\
{\small $^{2}$ Department of Mathematics, McMaster
University, Hamilton, Ontario, Canada, L8S 4K1}  }

\date{\today}
\maketitle

\begin{abstract}
Periodic waves of the one-dimensional cubic defocusing NLS equation
are considered.  Using tools from integrability theory, these waves
have been shown in \cite{Decon} to be linearly stable and the
Floquet--Bloch spectrum of the linearized operator has been explicitly
computed. We combine here the first four conserved quantities of the
NLS equation to give a direct proof that cnoidal periodic
waves are orbitally stable with respect to subharmonic perturbations,
with period equal to an integer multiple of the period of the wave.
Our result is not restricted to the periodic waves of small amplitudes.
\end{abstract}

\section{Introduction}
\label{sec:intro}

We consider the cubic defocusing NLS (nonlinear Schr\"{o}dinger)
equation in one space dimension:
\begin{equation}
\label{nls}
i \psi_t + \psi_{xx} - |\psi|^2 \psi = 0,
\end{equation}
where $\psi = \psi(x,t) \in \C$ and $(x,t) \in \R \times \R$. This
equation arises in the study of modulational stability of small
amplitude nearly harmonic waves in nonlinear dispersive systems
\cite{Sulem}. In this context, monochromatic waves of the original
system correspond to spatially homogeneous solutions of the cubic NLS
equation (\ref{nls}) of the form $\psi(x,t) = a e^{-i a^2 t}$, where
the positive parameter $a$ can be taken equal to one without loss of
generality, due to scaling invariance. According to the famous
Lighthill criterion, these plane waves are spectrally stable with
respect to sideband perturbations \cite{Ostrovsky}, because the
nonlinearity in (\ref{nls}) is defocusing. Moreover, using energy
methods, it can be shown that plane waves are also orbitally stable
under perturbations in $H^1(\R)$ \cite[Section~3.3]{Zhidkov}, where
the orbit is defined with respect to arbitrary rotations of the
complex phase of $\psi$.

More generally, it is important for the applications to consider
spatially inhomogeneous waves of the form $\psi(x,t) = u_0(x)
e^{-i t}$, where the profile $u_0 : \R \to \C$ satisfies the
second-order differential equation
\begin{equation}
\label{wave}
\frac{d^2 u_0}{d x^2} + (1 - |u_0|^2) u_0 = 0, \quad x \in \R.
\end{equation}
Such solutions of the cubic NLS equation (\ref{nls}) correspond to
slowly modulated wave trains of the original physical system.  A
complete list of all bounded solutions of the second-order equation
(\ref{wave}) is known, see \cite{Decon,GH2}. Most of them are
quasi-periodic in the sense that $u_0(x) = r(x)e^{i\phi(x)}$ for some
real-valued functions $r,\phi$ such that $r$ and $\phi'$ are periodic
with the same period $T_0 > 0$. The corresponding solutions of the
cubic NLS equation (\ref{nls}) are usually called ``periodic waves'',
although strictly speaking they are not periodic functions of $x$ in
general. In addition, the second-order equation (\ref{wave}) has
nonperiodic solutions such that $r$ and $\phi'$ converge to a limit as
$x \to \pm\infty$; these correspond to ``dark solitons'' of the cubic
NLS equation. In the present paper, we focus on real-valued solutions
of the second-order equation (\ref{wave}), which form a one-parameter
family of periodic waves (often referred to as ``cnoidal waves'').

Several recent works addressed the stability of periodic waves for
the cubic NLS equation (\ref{nls}). Using the energy method, it
was shown in \cite{GH1,GH2} that periodic waves are orbitally
stable within a class of solutions which have the same periodicity
properties as the wave itself. More precisely, if $u_0(x) =
e^{ipx} q_0(x)$ where $p \in \R$ and $q_0$ is $T_0$-periodic, the
wave $u_0(x)e^{-it}$ is orbitally stable among solutions of the form
$\psi(x,t) = e^{i(px-t)} q(x,t)$, where $q(\cdot,t) \in H^1_{\rm per}(0,T_0)$.
Here the orbit is defined with respect to translations in space
and rotations of the complex phase. The proof follows the general
strategy proposed in \cite{GSS} and relies on the fact that the
periodic wave is a constrained minimizer of the energy
\begin{equation}
\label{energy}
E(\psi) = \int_\I \Bigl[ |\psi_x|^2 + \frac{1}{2} (1 - |\psi|^2)^2
\Bigr] \dd x,
\end{equation}
subject to fixed values of the charge $Q$ and
the momentum $M$ given by
\begin{equation}
\label{charge}
Q(\psi) = \int_\I |\psi|^2 \dd x, \quad M(\psi) =
\frac{i}{2} \int_\I \Bigl(\bar{\psi} \psi_x - \psi \bar{\psi}_x
\Bigr) \dd x.
\end{equation}
Here $\I = (0,T_0)$. On the other hand, if we consider the more
general case of ``subharmonic perturbations'', which correspond to
$q(\cdot,t) \in H^1_{\rm per}(0,NT_0)$ for some integer $N \ge 2$,
then the second variation of $E$ at $u_0$ with $\I = (0,N T_0)$
contains additional negative eigenvalues, which cannot be eliminated
by restricting the energy to the submanifold where $Q$ and $M$ are
constant.

Generally speaking, in such an unfavorable energy configuration, there
is no chance to establish orbital stability using the standard energy
method \cite{Pava}. However, the cubic defocusing NLS equation can (at
least formally) be integrated using the inverse scattering transform
method, and it admits therefore a countable sequence of independent
conserved quantities. For instance, one can verify directly or with an
algorithmic computation (see \cite[Section 2.3]{Yang} for a review of
such techniques) that the higher-order functional
\begin{equation}
\label{energy-R}
R(\psi) = \int_\I \Bigl[ |\psi_{xx}|^2 + 3 |\psi|^2 |\psi_x|^2 +
\frac{1}{2} (\bar{\psi} \psi_x + \psi \bar{\psi}_x)^2 + \frac{1}{2} |\psi|^6 \Bigr] \dd x,
\end{equation}
is also invariant under the time evolution defined by (\ref{nls}).
These additional properties can be invoked to rescue the stability
analysis of periodic waves. Indeed, using the eigenfunctions of Lax
operators arising in the inverse scattering method, a complete set of
Floquet--Bloch eigenfunctions satisfying the linearization of the
cubic NLS equation (\ref{nls}) at the periodic wave with profile $u_0$
has been constructed in \cite{Decon}.  Moreover, it is shown in
\cite{Decon} that an appropriate linear combination of the energy $E$,
the charge $Q$, the momentum $M$, and the higher order quantity $R$
produces a functional for which the periodic wave with profile $u_0$
is a strict local minimizer, up to symmetries. This result holds for
$q(\cdot,t) \in H^2_{\rm per}(0,N T_0)$, for any $N \in \mathbb{N}$,
where $T_0$ is the period of $|u_0|$.  This easily implies that the
periodic wave with profile $u_0$ is orbitally stable with respect to
subharmonic perturbations.

The proof given in \cite{Decon} that any periodic wave
can be characterized as a local minimizer of a suitable higher-order
conserved quantity is not direct. Indeed, the authors
prove the positivity of the second variation at the periodic wave by
evaluating the corresponding quadratic form on the basis of the
Floquet--Bloch eigenfunctions associated with the linearized NLS
flow. These, however, are {\em not} the eigenfunctions of the
self-adjoint operator associated with the second variation itself,
which would be more natural to use in the present context. In
addition, many explicit computations are not transparent because they
rely on nontrivial properties of the Jacobi elliptic functions and
integrals that are used to represent the profile $u_0$ of the periodic
wave. This is why we feel that it is worth revisiting the problem
using more standard PDE techniques, which is the goal of the
present work.

The idea of using higher-order conserved quantities to solve delicate
analytical problems related to orbital stability of nonlinear waves in
integrable evolution equations has become increasingly popular in
recent years. Orbital stability of $n$-solitons in the Korteweg--de
Vries (KdV) and the cubic focusing NLS equations was established in
the space $H^n(\R)$ by combining the first $(n+1)$ conserved
quantities of these equations in \cite{MS93} and \cite{Kap06},
respectively.  For the modified KdV equation, orbital stability of
breathers in the space $H^2(\R)$ was established in \cite{Munoz} by
using two conserved quantities. For the massive Thirring model (a
system of nonlinear Dirac equations), orbital stability of solitary
waves was proved in the space $H^1(\R)$ with the help of the first
four conserved quantities \cite{PY}.

As already mentioned, we consider in this paper periodic waves of the
cubic defocusing NLS equation (\ref{nls}) which correspond to {\em
 real-valued} solutions of the second-order equation (\ref{wave}). In
that case, the second-order equation (\ref{wave}) can be integrated
once to obtain the first-order equation
\begin{equation}
\label{first-order}
\left( \frac{d u_0}{d x} \right)^2 =\, \frac{1}{2} \left[ (1 - u_0^2)^2
- \mathcal{E}^2 \right], \quad x \in \R,
\end{equation}
where the integration constant $\mathcal{E} \in [0,1]$ can be used to
parameterize all bounded solutions, up to translations. If $0 <
\mathcal{E} < 1$, we obtain a periodic solution which has the explicit
form
\begin{equation}
\label{sn-periodic}
u_0(x) = \sqrt{1-\mathcal{E}} ~\mathrm{sn}\biggl(x\sqrt{\frac{1+
\mathcal{E}}{2}}\,,\,\sqrt{\frac{1-\mathcal{E}}{1+\mathcal{E}}}
\biggr),
\end{equation}
where $\mathrm{sn}(\xi,k)$ denotes the Jacobi elliptic function with
argument $\xi$ and parameter $k$ \cite{Lawden}. This solution
corresponds to a closed orbit in the phase plane for $(u_0,u_0')$,
which is represented in Figure~\ref{fig1}. When $\mathcal{E} \to 1$
the orbit shrinks to the center point $(0,0)$, while in the limit
$\mathcal{E} \to 0$ the solution $u_0$ approaches the black soliton
\begin{equation}
\label{black-soliton}
u_0(x) = \tanh\left( \frac{x}{\sqrt{2}} \right),
\end{equation}
which corresponds to a heteroclinic orbit connecting the two saddle
points $(-1,0)$ and $(1,0)$. If $\mathcal{E} \in (0,1)$, the period
of $u_0$ (which is exactly twice the period $T_0$ of the modulus $|u_0|$)
is given by
\begin{equation}
\label{K-periodic}
2 T_0 = 4\sqrt{\frac{2}{1+\mathcal{E}}}\,K\biggl(\sqrt{\frac{1-
\mathcal{E}}{1+\mathcal{E}}}\biggr),
\end{equation}
where $K(k)$ is the complete elliptic integral of the first kind.
It can be verified that $T_0$ is a decreasing function of $\mathcal{E}$
which satisfies $T_0 \to +\infty$ as $\mathcal{E} \to 0$ and $T_0 \to
\pi$ as $\mathcal{E} \to 1$ \cite{GH2}.

\begin{figure}[h]
\begin{center}
\includegraphics[scale=0.4]{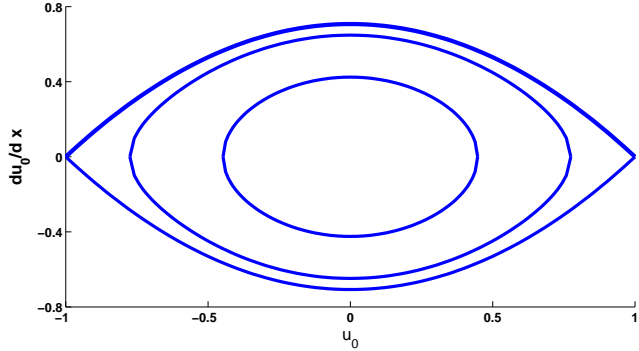}
\end{center}
\caption{\label{fig1}
The level set given by (\ref{first-order}) on the phase plane $(u_0,u_0')$
for $\mathcal{E} = 0,\, 0.4,\, 0.8$.}
\end{figure}

Now we study the stability of the periodic wave $\psi(x,t) =
u_0(x) e^{-it}$, where $u_0$ is given by (\ref{sn-periodic}) for
some $\mathcal{E} \in (0,1)$. It is clear from (\ref{wave}) that
the wave profile $u_0$ is a critical point of the energy functional
$E$ defined by (\ref{energy}). In addition, one can verify by
explicit (but rather cumbersome) calculations that $u_0$ is
also a critical point of the higher-order functional
\begin{equation}
\label{energy-S}
S(\psi) = R(\psi) - \frac{1}{2} (3 - \mathcal{E}^2) Q(\psi),
\end{equation}
where $R$ is given by (\ref{energy-R}) and $Q$ by (\ref{charge}).
Using an idea borrowed from \cite{Decon}, we combine $E$ and $S$
by introducing the functional
\begin{equation}
\label{Lyapunov-functional}
\Lambda_c(\psi) = S(\psi) - c E(\psi),
\end{equation}
where $c \in \R$ is a parameter that will be fixed below. Our
first result is the following proposition, which establishes an
{\em unconstrained} variational characterization for the periodic
waves of the NLS equation (\ref{nls}), at least when their amplitude
is small enough.

\begin{proposition}
\label{prop-per}
There exists $\mathcal{E}_0 \in (0,1)$ such that, for all $\mathcal{E}
\in (\mathcal{E}_0,1)$, there exist values $c_-$ and $c_+$ in the
range $1 < c_- < 2 < c_+ < 3$ such that, for any $c \in (c_-, c_+)$,
the second variation of the functional $\Lambda_c$ at the periodic
wave profile $u_0$ is nonnegative for perturbations in $H^2(\R)$.
Furthermore, we have
\begin{equation}
\label{asymptotics-c-plus-minus}
c_{\pm} = 2 \pm \sqrt{2 (1 - \mathcal{E})} + \mathcal{O}(1 -
\mathcal{E}) \quad
\mbox{\rm as} \quad \mathcal{E} \to 1.
\end{equation}
\end{proposition}

\begin{remark}
\label{rem-per1}
The second variation of $\Lambda_c$ at $u_0$ is the quadratic form
associated with a fourth-order selfadjoint operator with
$T_0$-periodic coefficients, which will be explicitly calculated in
Section~\ref{sec:small} below. Proposition~\ref{prop-per} asserts that
the Floquet--Bloch spectrum of that operator is nonnegative, if we
consider it as acting on the whole space $H^4(\R)$. In particular, the
same operator has nonnegative spectrum when acting on $H^4_{\rm
  per}(0,T)$, where $T$ is any multiple of $T_0$.  In fact, the proof
of Proposition~\ref{prop-per} shows that $\Lambda_c''(u_0)$ is
positive except for two neutral directions corresponding to symmetries
(translations in space and rotations of the complex phase). This key
observation will allow us to prove orbital stability of the periodic wave with respect to
subharmonic perturbations, see Theorem \ref{theorem-per} below.
\end{remark}

Our second result suggests a rather explicit formula for the
limiting values $c_{\pm}$ that appear in Proposition \ref{prop-per}.

\begin{proposition}
\label{prop-explicit}
For all $\mathcal{E} \in (0,1)$ and all $c \ge 1$, the second variation
of the functional $\Lambda_c$ at the periodic wave profile $u_0$ is
positive, except for two neutral directions due to symmetries, only
if $c \in [c_-,c_+]$ with
\begin{equation}
\label{exact-c-plus-minus}
c_{\pm} = 2 \pm \frac{2 k}{1 + k^2}, \qquad
\hbox{where} \quad k = \sqrt{\frac{1-\mathcal{E}}{1+\mathcal{E}}}.
\end{equation}
\end{proposition}

\begin{remark}
\label{rem-per2}
Proposition \ref{prop-explicit} gives a necessary condition for the
second variation $\Lambda_c''(u_0)$ to be positive except for two
neutral directions due to translations and phase rotations. The
condition is obtained by considering one particular band of the
Floquet--Bloch spectrum of the fourth-order operator associated with
$\Lambda_c''(u_0)$. That band touches the origin when the
Floquet-Bloch wave number is equal to zero, is strictly convex near
the origin if $c \in (c_-,c_+)$, and strictly concave if $c \ge 1$
and $c \notin [c_-,c_+]$. In the latter case, the second variation
$\Lambda_c''(u_0)$ has therefore negative directions. Interestingly
enough, the alternative approach of Bottman {\em et al.} \cite{Decon}
suggests that, for any $\mathcal{E} \in (0,1)$, the second variation
$\Lambda_c''(u_0)$ is positive (except for neutral directions due to
symmetries) whenever $c \in (c_-,c_+)$. Indeed, after adopting our
definition of the functionals $E$ and $S$, and performing explicit
computations with Jacobi elliptic functions, one can show that
the conditions implicitly defined in \cite[Theorem 7]{Decon}
exactly correspond to choosing our parameter $c$ in the interval
$(c_-,c_+)$ given by (\ref{exact-c-plus-minus}).
\end{remark}

In Figure~\ref{fig3}, the values $c_{\pm}$ are represented as a
function of the parameter $\mathcal{E}$ by a solid line. Note that
the asymptotic expansion (\ref{asymptotics-c-plus-minus}) is recovered
from the analytical expressions (\ref{exact-c-plus-minus}) in the
limit $k \to 0$, that is, $\mathcal{E} \to 1$.  The asymptotic result
(\ref{asymptotics-c-plus-minus}) is shown by dashed lines.

\begin{figure}[h]
\begin{center}
\includegraphics[scale=0.4]{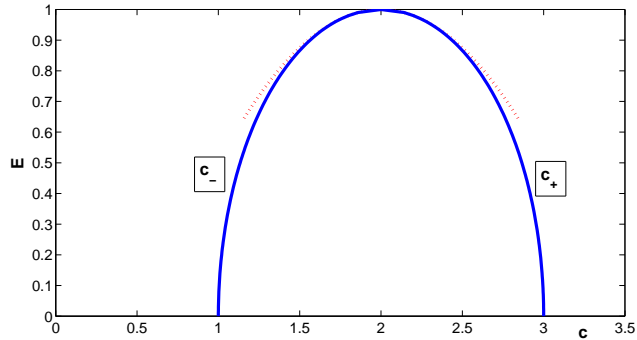}
\end{center}
\caption{\label{fig3}
The values $c_{\pm}$ given by the explicit expressions
(\ref{exact-c-plus-minus}) are represented as a function of the
parameter $\mathcal{E}$ (solid line). The asymptotic result
(\ref{asymptotics-c-plus-minus}) is shown by dashed lines.}
\end{figure}

The result of Proposition \ref{prop-per} relies on perturbation theory and
is therefore restricted to periodic waves of small amplitude. Although
the analytic formula (\ref{exact-c-plus-minus}) suggests that the conclusion of Proposition
\ref{prop-per} should hold for all periodic waves, namely for all
$\mathcal{E} \in (0,1)$, the result of Proposition \ref{prop-explicit}
is only a necessary condition for positivity of the functional $\Lambda_c$.
In the next result, we fix $c = 2$ (the mean value in the interval $[c_-,c_+]$)
and prove the positivity of the second variation of the functional $\Lambda_{c=2}$.

\begin{proposition}
\label{prop-positivity}
Fix $c = 2$. For every $\mathcal{E} \in (0,1)$, the second variation
of the functional $\Lambda_{c=2}$ at the periodic wave profile $u_0$
is positive, except for two neutral directions due to symmetries.
\end{proposition}

\begin{remark}
In the proof of Proposition \ref{prop-positivity}, we show that the quadratic form
defined by the second variation $\Lambda_{c=2}''(u_0)$ restricted to purely
imaginary perturbations of the periodic wave can be decomposed as a sum of
squared quantities, hence is obviously nonnegative. In order to control the quadratic
form for the real perturbations to the periodic wave,
we use a continuation argument from the limit to the periodic waves of small amplitude,
combined with analysis of the second-order Schr\"{o}dinger operators with $T_0$-periodic
coefficients.
\end{remark}

\begin{remark}
Proposition \ref{prop-positivity} implies the spectral stability of the periodic wave
profile $u_0$ for every $\mathcal{E} \in (0,1)$, see the end of Section \ref{sec:operator}.
\end{remark}

Our final result establishes orbital stability of the periodic wave
(\ref{sn-periodic}) with respect to the subharmonic perturbations in $H^2_{\rm per}(0,T)$,
where $T > 0$ is any integer multiple of the period of $u_0$. Therefore, we use
$\I = (0,T)$ in the definition of all functionals
(\ref{energy})-(\ref{energy-R}). If we consider $\Lambda_c$ as defined
on $H^2_{\rm per}(0,T)$, we know from Proposition~\ref{prop-positivity}
that $\Lambda_c'(u_0) = 0$ and that the second
variation $\Lambda_{c=2}''(u_0)$ is strictly positive, except for two
neutral directions corresponding to symmetries. Since $\Lambda_{c=2}(\psi)$
is a conserved quantity under the evolution defined by the cubic
NLS equation (\ref{nls}), we obtain the following orbital stability result.

\begin{theorem}
\label{theorem-per}
Fix $\mathcal{E}\in (0,1)$ and let $T$ be an integer multiple
of the period $2T_0$ of $u_0$. For any $\epsilon > 0$, there exists
$\delta > 0$ such that, if $\psi_0 \in H^2_{\rm per}(0,T)$
satisfies
\begin{equation}
\label{bound-initial-per}
\|\psi_0 - u_0\|_{H^2_{\rm per}(0,T)} \le \delta,
\end{equation}
the unique global solution $\psi(\cdot,t)$ of the cubic NLS equation
(\ref{nls}) with initial data $\psi_0$ has the following property.
For any $t \in \R$, there exist $\xi(t) \in \R$ and $\theta(t) \in
\R/(2\pi\Z)$ such that
\begin{equation}
\label{bound-final-per}
\|e^{i (t + \theta(t))}\psi(\cdot + \xi(t),t) - u_0\|_{H^2_{\rm per}(0,T)}
\le \epsilon.
\end{equation}
Moreover $\xi$ and $\theta$ are continuously differentiable
functions of $t$ which satisfy
\begin{equation}
\label{bound-time-per}
|\dot \xi(t)| + |\dot \theta(t)| \le C \epsilon, \quad t \in \R,
\end{equation}
for some positive constant $C$.
\end{theorem}

\begin{remark}
\label{cauchy-per}
It is well known that the Cauchy problem for the cubic NLS equation
(\ref{nls}) is globally well posed in the Sobolev space
$H^s_{\rm per}(0,T)$ for any integer $s \ge 0$, see \cite{Bour}.
\end{remark}

\begin{remark}
\label{limitation-rem}
The proof of Theorem~\ref{theorem-per} shows that, when $\epsilon
\le 1$, one can take $\delta = \epsilon/\mathcal{C}$ for some
constant $\mathcal{C} \ge 1$ depending on $\mathcal{E}$ and
on the ratio $T/T_0$. We emphasize, however, that $\mathcal{C} \to
\infty$ as $T/T_0 \to \infty$. This indicates that, although a given
periodic wave is orbitally stable with respect to perturbations
with arbitrary large period $T$, the size of the
stability basin becomes very small when the ratio $T/T_0$ is large.
\end{remark}

Applying the same technique, we can also prove the orbital stability of the
black soliton (\ref{black-soliton}) with respect to perturbations in $H^2(\mathbb{R})$.
The details of this analysis are given in Part II, which is a companion paper
to this work.

The rest of this article is organized as follows.
Section~\ref{sec:small} contains the proof of Proposition~\ref{prop-per}.
The sufficient condition of Proposition \ref{prop-explicit} is proved
in Section \ref{sec-positivity}. In Section~\ref{sec:positive}, we
provide a representation of the quadratic form associated with
$\Lambda_{c}''(u_0)$ as a sum of squared quantities.
Section \ref{sec:operator} reports the continuation
argument, which yields the proof of Proposition \ref{prop-positivity}.
Section~\ref{sec:periodic} is devoted to the proof of Theorem~\ref{theorem-per}.
Appendix A summarizes some explicit computations
with the use of Jacobi elliptic functions.

\section{Positivity of $\Lambda_c''(u_0)$ for periodic waves of
small amplitude}
\label{sec:small}

This section presents the proof of Proposition~\ref{prop-per}.

Let $u_0$ be the periodic wave profile defined by
(\ref{sn-periodic}) for some $\mathcal{E} \in (0,1)$. We consider
perturbations of $u_0$ of the form $\psi = u_0 + u + i v$, where $u,v$
are real-valued. Since $u_0$ is a
critical point of both $E$ and $S$ defined by (\ref{energy}) and (\ref{energy-S}),
the leading order contributions to
the renormalized quantities $E(\psi) - E(u_0)$ and $S(\psi) - S(u_0)$
are given by the second variations
\begin{equation}
\label{secondE}
\half \langle E''(u_0)[u,v], [u,v]\rangle \,=\,
\int_\I \left[ u_x^2 + (3 u_0^2 - 1) u^2 \right] \dd x + \int_\I
\left[ v_x^2 + (u_0^2 - 1) v^2 \right] \dd x
\end{equation}
and
\begin{align}
\nonumber
\half \langle S''(u_0)[u,v], [u,v]\rangle \,&=\,
\int_\I \left[ u_{xx}^2 + 5 u_0^2 u_x^2 + (-5 u_0^4 + 15 u_0^2 - 4 +
3 \mathcal{E}^2) u^2 \right] \dd x \\ \label{secondS}
&\,\quad + \int_\I \left[ v_{xx}^2 + 3 u_0^2 v_x^2 + (u_0^2 - 1)
v^2 \right] \dd x.
\end{align}

In the proof of the orbital stability theorem (Theorem \ref{theorem-per})
given in Section~\ref{sec:periodic}, we eventually take $\I = (0,T)$, where $T$ is a multiple of
the period $2T_0$ of the periodic wave profile $u_0$, and we assume that $u,v \in H^2_{\rm per}(0,T)$.
In this case, the formulas (\ref{secondE}) and (\ref{secondS}) represent the second variations of
the functionals $E$ and $S$ defined on the space $H^2_{\rm per}(0,T)$. However, here
and in the following three sections, we only investigate the positivity properties of the
second variations. For that purpose, it is more convenient to take $\I = \R$ and to assume that $u,v \in H^2(\R)$.

As is clear from (\ref{secondE}) and (\ref{secondS}), the second
variations $E''(u_0)$ and $S''(u_0)$ are block-diagonal in the sense
that the contributions of $u$ and $v$ do not mix together (this is the
main reason for which we restrict our analysis to {\em real-valued}
wave profiles $u_0$). We can thus write
\begin{equation*}
\half \langle E''(u_0)[u,v], [u,v]\rangle \,=\,
\langle L_+ u,u\rangle_{L^2} + \langle L_- v,v\rangle_{L^2}
\end{equation*}
and
\begin{equation*}
\half \langle S''(u_0)[u,v], [u,v]\rangle \,=\,
\langle M_+ u,u\rangle_{L^2} + \langle M_- v,v\rangle_{L^2},
\end{equation*}
where $\langle\cdot\,,\cdot\rangle_{L^2}$ is the scalar product
on $L^2(\R)$ and the operators $L_\pm$ and $M_\pm$ are
defined by
\begin{equation}
\label{operatorsdef}
\begin{array}{l}
L_+ \,=\, -\partial_x^2 + 3 u_0^2 - 1, \\[1mm]
L_- \,=\, -\partial_x^2 + u_0^2 - 1,
\end{array} \qquad
\begin{array}{l}
M_+ \,=\,  \partial_x^4 - 5 \partial_x u_0^2 \partial_x -5 u_0^4 +
15 u_0^2 - 4 + 3 \mathcal{E}^2, \\[1mm]
M_- \,=\, \partial_x^4 - 3 \partial_x u_0^2 \partial_x + u_0^2 - 1.
\end{array}
\end{equation}
Note that  $L_+ u_0' = M_+ u_0' = 0$, due to
the translation invariance of the cubic NLS equation (\ref{nls}),
and that $L_- u_0 = M_- u_0 = 0$, due to the gauge invariance
$\psi \mapsto e^{i\theta}\psi$ with $\theta \in \mathbb{R}$.

We now fix $c \in \R$ and consider the functional $\Lambda_c(\psi) =
S(\psi) - c E(\psi)$, as in (\ref{Lyapunov-functional}). We have
\begin{equation}
\label{Lambdasecond}
\half \langle \Lambda_c''(u_0)[u,v], [u,v]\rangle \,=\,
\langle K_+(c) u,u\rangle_{L^2} + \langle K_-(c) v,v\rangle_{L^2},
\end{equation}
where $K_{\pm}(c) = M_{\pm} - c L_{\pm}$. By construction, $K_\pm(c)$
are selfadjoint, fourth-order differential operators on $\R$ with
$T_0$-periodic coefficients, where $T_0$ is the period of $|u_0|$.
Our goal is to show that these operators are nonnegative, at least if
$\mathcal{E}$ is sufficiently close to $1$ and if the parameter $c$
is chosen appropriately. Equivalently, the quadratic forms in the
right-hand side of (\ref{Lambdasecond}) are nonnegative for all
$u,v \in H^2(\R)$ under the same assumptions on $\mathcal{E}$ and $c$.

Before going further, let us explain why a careful choice of the
parameter $c$ is necessary. Assume for simplicity that
$\mathcal{E} = 1$, so that $u_0 = 0$. In that case, we have
\begin{align}
\nonumber
\langle K_\pm(c) u,u\rangle_{L^2} \,&=\, \int_\R \left[ u_{xx}^2 - c u_x^2 +
(c-1) u^2 \right] \dd x \\
\label{zero-solution}
\,&=\, \int_\R \left( u_{xx} + \frac{c}{2}
u \right)^2 \dd x - \left(1 -\frac{c}{2} \right)^2 \int_\R u^2 \dd x.
\end{align}
This simple computation shows that the second variation
$\Lambda_c''(0)$ is nonnegative if and only if $c = 2$.
By a perturbation argument, we shall verify that $\Lambda_c''(u_0)$
remains nonnegative for $\mathcal{E}$ sufficiently close to $1$,
provided $c$ is close enough to $2$. More precisely, we shall prove
that the operators $K_+(c)$ and $K_-(c)$ are nonnegative and have
only the following zero modes
 \begin{equation}
 \label{zero-eigenfunctions}
 K_+(c) u_0' = 0, \quad \mbox{\rm and} \quad K_-(c) u_0 = 0.
 \end{equation}
This means that the second variation $\Lambda''_c(u_0)$
is strictly positive, except along the subspace spanned by the
eigenfunctions $u_0'$ and $i u_0$, which correspond to symmetries
of the NLS equation (\ref{nls}). Note that, when $\mathcal{E} = 1$,
the second variation $\Lambda''_c(u_0)$ vanishes on a four-dimensional subspace,
according to the representation (\ref{zero-solution}), but the
degeneracy disappears as soon as $\mathcal{E} < 1$.

The proof of Proposition \ref{prop-per} relies on perturbation theory
for the Floquet--Bloch spectrum of the operators $K_{\pm}(c)$. First,
we normalize the period of the profile $u_0$ to $2\pi$ by using the
transformation $u_0(x) = U(\ell x)$, where $\ell = \pi/T_0$, so
that $U(z+2\pi) = U(z)$. The second-order differential equation
satisfied by rescaled profile $U(z)$, as well as the associated
first-order invariant, are given by
\begin{equation}
\label{wave-problem}
\ell^2 \frac{d^2 U}{d z^2} + U - U^3 = 0 \quad \Rightarrow \quad
\ell^2 \left( \frac{d U}{d z} \right)^2 = \frac{1}{2} \left[
(1 - U^2)^2 - \mathcal{E}^2 \right].
\end{equation}
In agreement with the exact solution (\ref{sn-periodic}) we assume
that $U$ is odd with $U'(0) > 0$, so that $U \in H^2_{\rm per}(0,2\pi)$
is entirely determined by the value of $\mathcal{E}
\in (0,1)$. As was already mentioned, it is known for the soft
potential in (\ref{wave-problem}) that the map $(0,1) \ni \mathcal{E}
\mapsto \ell \in (0,1)$ is strictly increasing and onto
\cite{GH2}. The following proposition specifies the precise asymptotic
behavior of the rescaled profile $U$ as $\mathcal{E} \to 1$.

\begin{proposition}
The map $(0,1) \ni \mathcal{E} \mapsto (\ell,U) \in \R \times
H^2_{\rm per}(0,2\pi)$ can be uniquely described, when
$\mathcal{E} \to 1$, by a small parameter $a > 0$ in the
following way:
\begin{equation}
\label{proposition-expansion}
\mathcal{E} = 1 - a^2 + \mathcal{O}(a^4), \quad
\ell^2 = 1 - \frac{3}{4} a^2 + \mathcal{O}(a^4), \quad
U(z) = a U_0(z) + \mathcal{O}_{H^2_{\rm per}(0,2\pi)}(a^3),
\end{equation}
where $U_0(z) = \sin(z)$. \label{prop-expansion}
\end{proposition}

\begin{proof}
The argument is rather standard, so we just mention here the
main ideas. Since the wave profile $U(z)$ is an odd function of $z$,
we work in the space
$$
L^2_{{\rm per},{\rm odd}}(0,2\pi) = \{U \in L^2_{\rm loc}(\R) : \quad
U \hbox{ is odd and } 2\pi\hbox{--periodic}\}.
$$
We use the Lyapunov--Schmidt decomposition $\ell^2 = 1 +\tilde{\ell}$,
$U = a U_0 + \tilde{U}$, where the perturbation $\tilde U \in
H^2_{{\rm per},{\rm odd}}(0,2\pi)$ is orthogonal to $U_0$ in $L^2_{\rm per}
(0,2\pi)$, namely $\langle U_0,\tilde{U}\rangle_{L^2_{\rm per}} = 0$.
The quantities $\tilde\ell$ and $\tilde U$ can be determined by
projecting equation (\ref{wave-problem}) onto the one-dimensional
subspace ${\rm Span}\{U_0\} \subset L^2_{{\rm per},{\rm odd}}(0,2\pi)$
and its orthogonal complement. This gives the relations
\begin{equation}
\label{proj1}
a \tilde{\ell} = -\frac{\langle U_0, (a U_0 + \tilde{U})^3
\rangle_{L^2_{\rm per}}}{\langle U_0,U_0 \rangle_{L^2_{\rm per}}},
\end{equation}
and
\begin{equation}
\label{proj2}
(1+\tilde{\ell}) \tilde{U}'' + \tilde{U} = (a U_0 + \tilde{U})^3 -
\frac{\langle U_0, (a U_0 + \tilde{U})^3\rangle_{L^2_{\rm per}}}{
\langle U_0, U_0\rangle_{L^2_{\rm per}}}\,U_0.
\end{equation}
For any small $\tilde \ell$ and $a$, it is easy to verify (by
inverting the linear operator in the left-hand side and using a fixed
point argument) that equation (\ref{proj2}) has a unique solution
$\tilde{U} \in H^2_{{\rm per},{\rm odd}}(0,2\pi)$ such that $\langle
U_0,\tilde{U}\rangle_{L^2_{\rm per}} = 0$ and $\tilde{U} =
\mathcal{O}_{H^2_{\rm per}(0,2\pi)}(a^3)$ as $a \to 0$. This solution
depends smoothly on $\tilde\ell$, so if we substitute it into
the right-hand side of (\ref{proj1}) we obtain an equation for
$\tilde\ell$ only, which can in turn be solved uniquely for small
$a > 0$. The result is
$$
\tilde{\ell} = -a^2 \frac{\langle U_0, U_0^3\rangle_{L^2_{\rm per}}}{
\langle U_0, U_0\rangle_{L^2_{\rm per}}}  + \mathcal{O}(a^4) =
-\frac{3}{4} a^2 + \mathcal{O}(a^4).
$$
Finally the expression $\mathcal{E} = 1 - a^2 + \mathcal{O}(a^4)$
follows from the first-order invariant
(\ref{wave-problem}), if we use the above decompositions
and the asymptotic formulas for $\tilde\ell$
and $\tilde U$.
\end{proof}

We next study the Floquet--Bloch spectrum of the operators $K_{\pm}(c) = M_{\pm} - cL_{\pm}$. 
Using the same rescaling $z = \ell x$ and the Floquet parameter $\kappa$, 
we write these operators in the following form
\begin{align*}
P_-(c,\kappa) \,&=\, \ell^4 (\partial_z + i \kappa)^4 - 3 \ell^2
(\partial_z + i\kappa)U^2 (\partial_z + i \kappa)
+ c \ell^2 (\partial_z + i \kappa)^2 + (c-1) (1 - U^2), \\
P_+(c,\kappa) \,&=\,  \ell^4 (\partial_z + i \kappa)^4 - 5 \ell^2
(\partial_z + i\kappa)U^2 (\partial_z + i \kappa)
+ c \ell^2 (\partial_z + i \kappa)^2 \\
&\quad\, -5 U^4 + (15-3c) U^2 - 4 + 3
\mathcal{E}^2 + c.
\end{align*}
Note that the operators $P_{\pm}(c,\kappa)$ have $\pi$-periodic coefficients, hence
we can look for $\pi$-periodic Bloch wave functions so that $\kappa$ can be 
defined in the Brillouin zone $[-1,1]$. However, for computational
simplicity of the perturbation expansions, it is more convenient
to work with the $2\pi$-periodic Bloch wave functions, when 
$\kappa$ is defined in the Brillouin zone
$\mathbb{T} = \left[-\frac{1}{2},\frac{1}{2}\right]$. 
If $\kappa \in \mathbb{T}$ and if the function $w(\cdot,\kappa) \in
H^4_{\rm per}(0,2\pi)$ satisfies
\begin{equation}
\label{Floquet-spectrum}
P_{\pm}(c,\kappa) w(z,\kappa) = \lambda(\kappa) w(z,\kappa), \quad z \in \R,
\end{equation}
for some $\lambda(\kappa) \in \R$ and either sign,
then defining $u(x,\kappa) = e^{i \kappa \ell x} w(\ell x,\kappa)$ we obtain
a function $u(\cdot,\kappa) \in L^\infty(\R) \cap H^4_{\rm loc}(\R)$
such that
$$
K_{\pm}(c) u(x,\kappa) = \lambda(\kappa) u(x,\kappa), \quad x \in \R.
$$
This precisely means that $\lambda(\kappa)$ belongs to the Floquet--Bloch
spectrum of $K_{\pm}(c)$.

By Proposition~\ref{prop-expansion}, when $\mathcal{E}$ is
close to $1$, the operators $P_\pm(c,\kappa)$ can be expanded as
$$
P_{\pm}(c,\kappa) = P^{(0)}(c,\kappa) + a^2 P_{\pm}^{(1)}(c,\kappa) +
\mathcal{O}_{H^4_{\rm per}(0,2\pi) \to L^2_{\rm per}(0,2\pi)}(a^4),
$$
where
\begin{align*}
P^{(0)}(c,\kappa) \,&=\,  (\partial_z + i \kappa)^4 +
c (\partial_z + i \kappa)^2 + c - 1, \\
P_{-}^{(1)}(c,\kappa) \,&=\, -\frac{3}{2} (\partial_z + i \kappa)^4
- 3 (\partial_z + i\kappa) U_0^2 (\partial_z + i \kappa)
- \frac{3}{4} c (\partial_z + i \kappa)^2 + (1-c) U_0^2, \\
P_{+}^{(1)}(c,\kappa) \,&=\, -\frac{3}{2} (\partial_z + i \kappa)^4
- 5 (\partial_z + i\kappa) U_0^2 (\partial_z + i \kappa)
- \frac{3}{4} c  (\partial_z + i \kappa)^2 + (15-3c) U_0^2 - 6.
\end{align*}

The operator $P^{(0)}(c,\kappa)$ has constant coefficients, and its spectrum
in the space $L^2_{\rm per}(0,2\pi)$ consists of a countable family of
real eigenvalues $\{\lambda_n^{(0)}(\kappa) \}_{n \in \mathbb{Z}}$ given by
\begin{equation}
\label{lambda-n}
\lambda_n^{(0)}(\kappa) = (\kappa+n)^4 - c (\kappa+n)^2 + c - 1,
\quad n \in \mathbb{Z}.
\end{equation}
As was already observed, one has $\lambda_n^{(0)}(\kappa) \ge 0$ for all $n
\in \mathbb{Z}$ and all $\kappa \in \mathbb{T}$ if and only if $c = 2$.
This is the case represented in Figure~\ref{fig2} (left), where it
is clear that all spectral bands $\{\lambda_n^{(0)}(\kappa)\}_{\kappa \in \mathbb{T}}$
are strictly positive, except for two bands corresponding to
$n = \pm 1$ which touch the origin at $\kappa = 0$.

\begin{figure}[h]
\begin{center}
\includegraphics[scale=0.32]{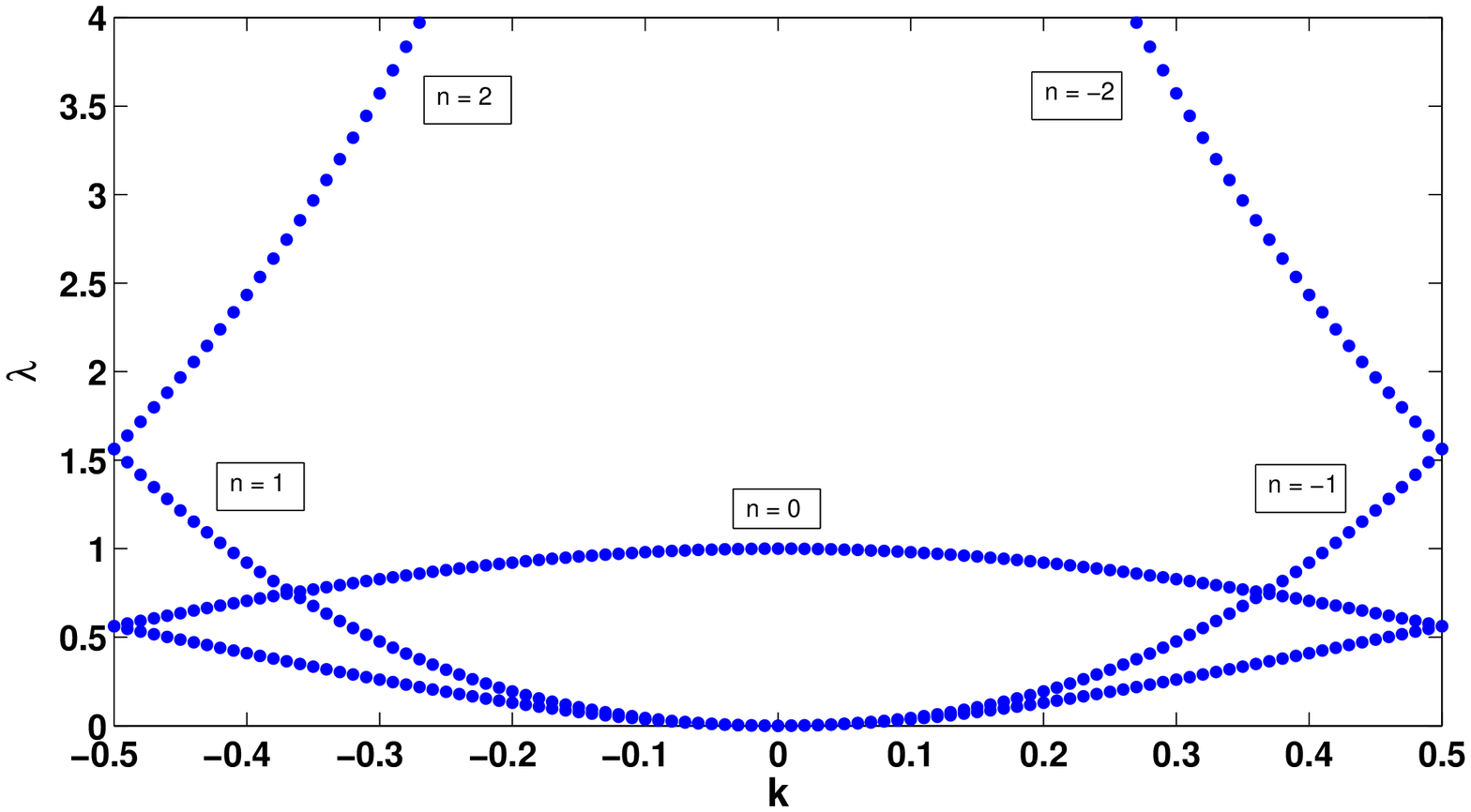}
\includegraphics[scale=0.32]{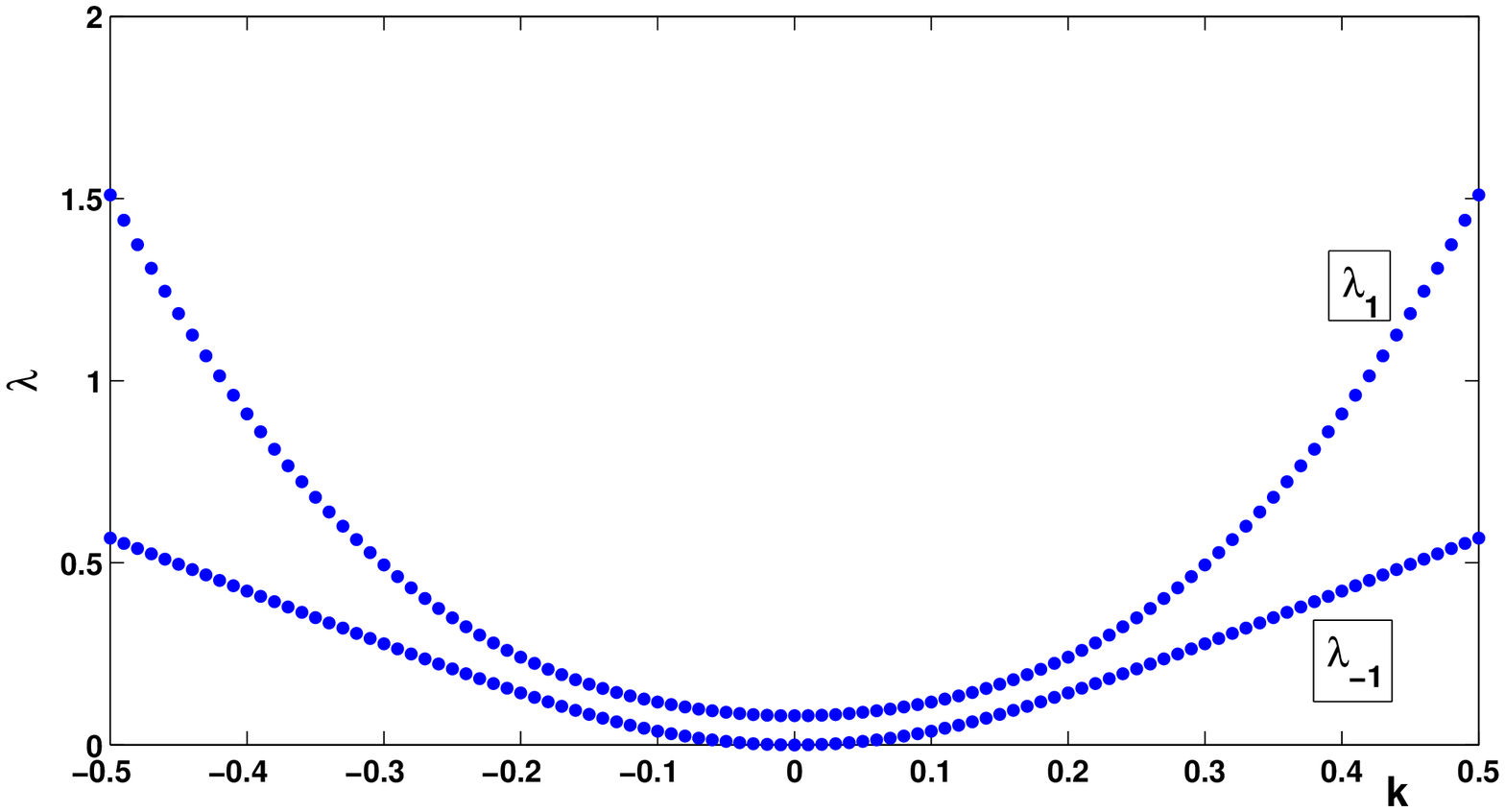}
\end{center}
\caption{\label{fig2}
Left: spectral bands given by (\ref{lambda-n})
for $c = 2$ and $a = 0$. Right: spectral bands given by
the matrix eigenvalue problem (\ref{perturbed-matrix}) for $c = 2$
and $a = 0.2$.}
\end{figure}

For small $a > 0$, the eigenvalues of the perturbed operators
$P_\pm(c,\kappa)$ are denoted by $\lambda_n^{\pm}(\kappa)$ with $n \in
\mathbb{Z}$, and we number them in such a way that
$\lambda_n^\pm(\kappa) \to \lambda_n^{(0)}(\kappa)$ as $a \to 0$ for fixed
$\kappa \in \mathbb{T}$.  By classical perturbation theory, we know
that the eigenvalues $\lambda_n^\pm(\kappa)$ stay bounded away from
zero for $n \neq \pm 1$, so it remains to study how the bands
$\{\lambda_1^\pm(\kappa)\}_{\kappa \in \mathbb{T}}$ and
$\{\lambda_{-1}^\pm(\kappa) \}_{\kappa \in \mathbb{T}}$ behave near
$\kappa = 0$ as $a \to 0$.  The following proposition indicates that
these bands separate from each other when $a > 0$, so that one band
still touches the origin at $\kappa = 0$ while the other one remains
strictly positive for all $\kappa \in \mathbb{T}$. In other words, the
degeneracy of the limiting case $c = 2$, $a = 0$ is unfold by the
perturbation as soon as $a > 0$. This phenomenon is illustrated in
Figure~\ref{fig2} (right), which shows the solutions of the matrix
eigenvalue problem (\ref{perturbed-matrix}) obtained below.

\begin{proposition}
\label{proposition-eigenvalues}
If $a > 0$ is sufficiently small and $c \in (c_-,c_+)$, where
\begin{equation}
\label{expansion-c}
c_{\pm} = 2 \pm \sqrt{2} a + \mathcal{O}(a^2),
\end{equation}
the operator $K_{\pm}(c)$ has exactly one Floquet--Bloch band
denoted by $\{\lambda_{-1}^\pm(\kappa)\}_{\kappa \in \mathbb{T}}$ that touches
the origin at $\kappa = 0$, while all other bands are strictly positive.
Moreover, for any $\nu < \sqrt{2}$, there exist positive constants
$C_1, C_2, C_3$ (independent of $a$) such that, if $|c-2| \le \nu a$,
one has
\begin{equation}
\label{expansion-lambda-bounds}
\lambda_{-1}^\pm(\kappa) \ge C_1 \kappa^2, \quad
\lambda_1^\pm(\kappa) \ge C_2 (a^2+\kappa^2), \quad \hbox{and}\quad
\lambda_n^\pm(\kappa) \ge C_3, \quad n \in \mathbb{Z} \setminus \{+1,-1\},
\end{equation}
for all $\kappa \in \mathbb{T}$.
\end{proposition}

\begin{proof}
From (\ref{lambda-n}) we know that, if $|c - 2|$ is sufficiently
small, there exists a constant $C > 0$ (independent of $c$)
such that
\begin{equation}
\label{resolvent}
0 < \lambda_n^{(0)}(\kappa)^{-1} \leq C \quad \hbox{for all } n \in \mathbb{Z}
\setminus\{+1,-1\} \hbox{ and all } \kappa \in \mathbb{T}.
\end{equation}
By classical perturbation theory, this bound remains true (with
possibly a larger constant $C$) for the perturbed eigenvalues
$\lambda_n^\pm(\kappa)$ when $n \neq \pm 1$ and $a > 0$ is small enough.
We thus obtain the third estimate in (\ref{expansion-lambda-bounds}).

To control the critical bands corresponding to $n = \pm 1$, we
concentrate on the operator $P_-(c,\kappa)$ (the argument for
$P_+(c,\kappa)$ being similar, see below), and for simplicity we
denote its eigenvalues by $\lambda_n(\kappa)$ instead of
$\lambda_n^{-}(\kappa)$. The same perturbation argument as before
shows that $\lambda_{\pm 1}(\kappa)$ is bounded away from zero if
$|\kappa| \ge \kappa_0$ and $a$ is sufficiently small, where $\kappa_0
> 0$ is an arbitrary positive number. On the other hand, for small
values of $a$, $|c-2|$, and $|\kappa|$, solutions to the spectral
problem (\ref{Floquet-spectrum}) for $P_-(c,\kappa)$ are obtained by
the Lyapunov--Schmidt decomposition
$$
w(z,\kappa) = b_1(\kappa) e^{iz} + b_{-1}(\kappa) e^{-iz} +
\tilde{w}(z,\kappa), \qquad \langle e^{\pm i \cdot},
\tilde{w}(\cdot,\kappa) \rangle_{L^2_{\rm per}} = 0,
$$
where all terms can be determined by projecting the spectral problem
(\ref{Floquet-spectrum}) onto the two-dimensional subspace ${\rm
Span}\{e^{i \cdot},e^{-i \cdot}\} \subset L^2_{\rm per}(0,2\pi)$ and
its orthogonal complement in $L^2_{\rm per}(0,2\pi)$. Using the bound
(\ref{resolvent}), one can prove that $\tilde{w}(\cdot,\kappa) =
\mathcal{O}_{ H^4_{\rm per}(0,2\pi)}(a^2)$, which allows us to find
$\lambda(\kappa)$ near $\lambda_{\pm 1}^{(0)}(\kappa)$ as a solution of
the matrix eigenvalue problem
\begin{eqnarray}
\nonumber
\left[ \begin{array}{cc} \lambda_{1}^{(0)}(\kappa) + a^2 g_{1,1}(\kappa)
+ \mathcal{O}(a^4) & a^2 g_{1,-1}(\kappa) + \mathcal{O}(a^4) \\
a^2 g_{-1,1}(\kappa) + \mathcal{O}(a^4) & \lambda_{-1}^{(0)}(\kappa)
+ a^2 g_{-1,-1}(\kappa) + \mathcal{O}(a^4) \end{array} \right]
\left[ \begin{array}{c} b_1 \\ b_{-1} \end{array} \right] \\
\label{perturbed-matrix}
= \lambda(\kappa) \left[ \begin{array}{c} b_1 \\ b_{-1} \end{array} \right],
\end{eqnarray}
where
\begin{align*}
g_{\pm 1,\pm 1}(\kappa) \,&=\, -\frac{3}{2}(\kappa\pm 1)^4 + \frac{3}{4} c
(\kappa \pm 1)^2 + \frac{1}{2} (1-c) + \frac{3}{2} (\kappa \pm 1)^2, \\
g_{\pm 1,\mp 1}(\kappa) \,&=\, \frac{1}{4} (1-c) + \frac{3}{4} (\kappa^2-1).
\end{align*}
Setting $c = 2 + \gamma$ with small $|\gamma|$, we have
\begin{align*}
\lambda_{\pm 1}^{(0)}(\kappa) \,&=\, \mp 2 \gamma \kappa + (4-\gamma) \kappa^2 \pm
4 \kappa^3 + \kappa^4, \\
g_{\pm 1,\pm 1}(\kappa) \,&=\, 1 + \frac{1}{4} \gamma \pm \frac{3}{2}
\gamma \kappa - 6 \kappa^2 + \frac{3}{4} \gamma \kappa^2 \mp 6 \kappa^3
- \frac{3}{2} \kappa^4, \\ g_{\pm 1,\mp 1}(\kappa) \,&=\, -1 -
\frac{1}{4} \gamma + \frac{3}{4} \kappa^2.
\end{align*}
If we denote by $A$ the matrix in the left-hand side of
(\ref{perturbed-matrix}), we thus obtain the expansions
\begin{align*}
\half {\rm tr}(A) \,&=\, a^2 + 4 \kappa^2 + \mathcal{O}((a^2+\kappa^2)
(|\gamma| + a^2 + \kappa^2)), \\
{\rm det}(A) \,&=\, (a^2 + 4 \kappa^2)^2 - a^4 - 4 \gamma^2 \kappa^2 +
\mathcal{O}((a^2+\kappa^2)^2(|\gamma| + a^2 + \kappa^2)).
\end{align*}
As a result, the eigenvalues $\lambda_{\pm 1}(\kappa)$ of $A$ satisfy
\begin{align}
\nonumber
\lambda_{\pm 1}(\kappa) \,&=\, a^2 + 4 \kappa^2 + \mathcal{O}((a^2+\kappa^2)
(|\gamma|+a^2+\kappa^2)) \\ \label{expansion-lambda}
&\quad\, \pm \sqrt{a^4 + 4 \gamma^2 \kappa^2 +
\mathcal{O}((a^2+\kappa^2)^2(|\gamma|+a^2+\kappa^2))}.
\end{align}
It remains to analyze (\ref{expansion-lambda}). If $a > 0$ is small,
we obviously have
$$
\lambda_1(\kappa) \geq a^2 + 4 \kappa^2 + \mathcal{O}((a^2+\kappa^2)
(|\gamma| + a^2 + \kappa^2)) > 0,
$$
which implies the second bound in (\ref{expansion-lambda-bounds}).  To
estimate $\lambda_{-1}(\kappa)$, we first consider the regime where $|\kappa|
\le a$. If $|\gamma| \leq \nu a$ for any $\nu > 0$ independently of
$a$, further expansion of (\ref{expansion-lambda}) yields
\begin{equation}
\label{lamminus}
\lambda_{-1}(\kappa) = \mu + 4 \kappa^2 - \frac{2 \gamma^2\kappa^2}{a^2} +
\mathcal{O}(\kappa^2(|\gamma| + a^2)),
\end{equation}
where $\mu = \mathcal{O}(a^2 (|\gamma| + a^2))$ does not depend
on $\kappa$. But since $K_-(c) u_0 = 0$ for any $c$, we must have
$\lambda_{-1}(0) = 0$ to all orders in $a$ and $\gamma$,
hence actually $\mu = 0$. Then (\ref{lamminus}) shows that
$\lambda_{-1}(\kappa)$ has a nondegenerate minimum at $\kappa = 0$ if and
only if
\begin{equation}
\label{gamcond}
\gamma^2 < 2 a^2 + \mathcal{O}(a^3).
\end{equation}
Since $\gamma = c - 2$, this yields expansion (\ref{expansion-c}) for
$c_\pm$. From now on, we assume that $|\gamma| \le \nu a$ for some
$\nu \in (0,\sqrt{2})$, so that the inequality (\ref{gamcond})
certainly holds if $a$ is sufficiently small. The expansion
(\ref{lamminus}) shows that if $|\kappa| \leq a$, then
$$
\lambda_{-1}(\kappa) = (4-2\nu^2)\kappa^2 + \mathcal{O}(\kappa^2
(|\gamma| + a^2)).
$$
On the other hand, if $|\kappa| \ge a$, we easily find from
(\ref{expansion-lambda}) that
$$
\lambda_{-1}(\kappa) \ge 4\kappa^2 - 2|\gamma| |\kappa| + \mathcal{O}
(\kappa^2(|\gamma| + \kappa^2)^{1/2}) \ge \kappa^2 + \mathcal{O}
(\kappa^2(|\gamma| + \kappa^2)^{1/2}),
$$
because $2|\gamma| |\kappa| \le \kappa^2 + \gamma^2 \le \kappa^2 + \nu^2 a^2
\le 3\kappa^2$. Altogether, we obtain the first estimate in
(\ref{expansion-lambda-bounds}).

The spectral problem (\ref{Floquet-spectrum}) for the operator
$P_+(c,\kappa)$ can be studied in a similar way and results in
the matrix eigenvalue problem (\ref{perturbed-matrix}) with
\begin{align*}
g_{\pm 1,\pm 1}(\kappa) \,&=\, -\frac{3}{2}(\kappa\pm 1)^4 + \frac{3}{4}
c (\kappa \pm 1)^2 + \frac{3}{2} (1-c) + \frac{5}{2} (\kappa \pm 1)^2, \\
g_{\pm 1,\mp 1}(\kappa) \,&=\, \frac{3}{4} (5-c) + \frac{5}{4} (\kappa^2-1).
\end{align*}
Although the matrix $A$ has now different entries, the leading
order terms for the quantities ${\rm tr}(A)$ and ${\rm det}(A)$
are unchanged, hence the eigenvalues $\lambda_{\pm 1}(\kappa)$ still
satisfy (\ref{expansion-lambda}). Consequently, the conclusion
remains true for $c$ in the same interval (\ref{expansion-c}).
\end{proof}

\begin{remark}
In view of expansion (\ref{proposition-expansion}),
Proposition~\ref{prop-per} is a direct consequence of
Proposition~\ref{proposition-eigenvalues}.
\end{remark}

\section{Necessary condition for positivity of $\Lambda_c''(u_0)$}
\label{sec-positivity}

This section presents the proof of Proposition \ref{prop-explicit}.

In Section \ref{sec:small}, we only considered small amplitude periodic waves
(\ref{sn-periodic}) with $\mathcal{E}$ close to $1$. To get some
information on the quadratic form $\Lambda_c''(u_0)$ for larger
periodic waves, we recall that, for any $\mathcal{E} \in (0,1)$ and
any $c \in \mathbb{R}$, the operators $P_{\pm}(c,\kappa)$ have at
least one Floquet--Bloch spectral band that touches the origin at
$\kappa = 0$, because we know from (\ref{zero-eigenfunctions}) that
the kernel of $P_{\pm}(c,0)$ in $L^2_{\rm per}(0,2\pi)$ is nontrivial.

In what follows, we focus on the operator $P_-(c,\kappa)$. Assuming that
$\mathrm{ker}(P_-(c,0))$ in $L^2_{\rm per}(0,2\pi)$ is
one-dimensional, we compute an asymptotic expansion as $\kappa \to 0$
of the unique Floquet--Bloch band that touches the origin at $\kappa =
0$. By Proposition \ref{proposition-eigenvalues}, the assumption on
$\mathrm{ker}(P_-(c,0))$ is satisfied at least for the periodic waves of small
amplitude, in which case the Floquet--Bloch band that touches
the origin is actually the lowest band $\lambda^-_{-1}(\kappa)$.

\begin{proposition}
\label{proposition-lowest-band}
Fix $\mathcal{E} \in (0,1)$ and assume that $U= u_0(\ell^{-1} \cdot)$
is the only $2\pi$-periodic solution of the homogeneous equation
$P_-(c,0) w = 0$ for some $c \in \mathbb{R}$. Denote by $\mu(\kappa)$
the Floquet--Bloch band of $P_-(c,\kappa)$ that touches zero at $\kappa = 0$.
Then $\mu$ is $C^2$ near $\kappa = 0$, $\mu(0) = \mu'(0) = 0$, and
\begin{equation}
\label{expansion-lowest-band}
\mu''(0) = \frac{2}{\| U \|_{L^2_{\rm per}}^2} \left[-4 \ell^4 (c{-}2)^2
\langle U', (P_-(c,0))^{-1} U' \rangle_{L^2_{\rm per}} + 3 \ell^4
\| U' \|_{L^2_{\rm per}}^2 + (3{-}c) \ell^2 \| U \|^2_{L^2_{\rm per}} \right],
\end{equation}
where $W = (P_-(c,0))^{-1} U'$ is uniquely defined under the orthogonality
condition $\langle U, W \rangle_{L^2_{\rm per}} = 0$.
\end{proposition}

\begin{proof}
We consider $P_-(c,\kappa)$ as a self-adjoint operator in $L^2_{\rm per}
(0,2\pi)$ with domain $H^4_{\rm per}(0,2\pi)$. As $\kappa \to 0$, we have
$$
P_-(c,\kappa) = P_0(c) + i \kappa P_1(c) - \kappa^2 P_2(c) +
\mathcal{O}_{H^4_{\rm per}(0,2\pi) \to L^2_{\rm per}(0,2\pi)}(\kappa^3),
$$
where
\begin{align*}
P_0(c) \,&=\, \ell^4 \partial_z^4 - 3 \ell^2 \partial_z U^2 \partial_z
  + c \ell^2 \partial_z^2 + (c-1) (1 - U^2), \\
P_1(c) \,&=\, 4 \ell^4 \partial_z^3 - 6 \ell^2 U^2 \partial_z
  - 6 \ell^2 U U' + 2 c \ell^2 \partial_z, \\
P_2(c) \,&=\, 6 \ell^4 \partial_z^4 - 3 \ell^2 U^2  + c \ell^2.
\end{align*}
We note that $P_0(c)$ and $P_2(c)$ are self-adjoint, whereas $P_1(c)$
is skew-adjoint. Under the assumptions of the proposition, the
Floquet--Bloch band $\mu(\kappa)$ that touches zero at $\kappa = 0$ is
separated from all the other bands of $P_-(c,\kappa)$ locally near
$\kappa = 0$. Thus, $\mu(\kappa)$ is smooth near $\kappa = 0$, and it
is possible to choose a nontrivial solution $w(z,\kappa)$ of the
eigenvalue equation $P_-(c,\kappa) w(z,\kappa) = \mu(\kappa)
w(z,\kappa)$ which also depends smoothly on $\kappa$. We look for an
expansion of the form
\begin{equation*}
\mu(\kappa) = i \kappa \mu_1 - \kappa^2 \mu_2 + \mathcal{O}(\kappa^3)
\end{equation*}
and
\begin{equation*}
w(z,\kappa) = U(z) + i\kappa w_1(z) - \kappa^2 w_2(z) +
\mathcal{O}_{H^4_{\rm per}(0,2\pi)}(\kappa^3),
\end{equation*}
where $w_1$, $w_2$, and the remainder term belong to the orthogonal
complement of ${\rm span}\{U\}$ in $L^2_{\rm per}(0,2\pi)$. This gives
the following system for the correction terms
\begin{align}
\label{first-eq} P_0(c) w_1 + P_1(c) U \,&=\, \mu_1 U, \\
\label{second-eq} P_0(c) w_2 + P_1(c) w_1 + P_2(c) U \,&=\, \mu_1 w_1
+ \mu_2 U.
\end{align}
If we take the scalar product of (\ref{first-eq}) with $U$ in
$L^2_{\rm per}(0,2\pi)$ and use the fact that $P_0(c)$ is self-adjoint,
$P_1(c)$ is skew-adjoint, and $P_0(c)U = 0$, we obtain $\mu_1 = 0$.
Similarly, taking the scalar product of (\ref{second-eq}) with $U$
gives a nontrivial equation for $\mu_2$:
$$
\mu_2 \| U \|_{L^2_{\rm per}}^2 = \langle U, P_1(c) w_1 \rangle_{L^2_{\rm per}}
+ \langle U, P_2(c) U \rangle_{L^2_{\rm per}}.
$$
We note that
\begin{align*}
P_1(c) U \,&=\, 2 \ell^2 (2 \ell^2 U''' - 6 U^2 U' + c U') =
2 \ell^2 (c-2) U', \\
P_2(c) U \,&=\, \ell^2 (6 \ell^2 U'' - 3 U^3 + c U) = \ell^2
(3 \ell^2 U'' + (c-3) U).
\end{align*}
Setting $w_1 = -2 \ell^2 (c-2) W$, where $W$ is the unique solution of
$P_0(c) W = U'$ subject to the orthogonality condition $\langle U,
W \rangle_{L^2_{\rm per}} = 0$, we obtain
$$
\mu_2 \| U \|_{L^2_{\rm per}}^2 = 4 \ell^4 (c-2)^2 \langle U', W
\rangle_{L^2_{\rm per}} - \left( 3 \ell^4 \| U' \|_{L^2_{\rm per}}^2
+ (3-c) \ell^2 \| U \|^2_{L^2_{\rm per}} \right).
$$
which yields the result (\ref{expansion-lowest-band}) since
$\mu''(0) = -2 \mu_2$.
\end{proof}

Note that the first term in the right-hand side of
(\ref{expansion-lowest-band}) is negative, whereas the other two are
positive for $c \leq 3$. In the particular case where $c = 2$, it follows from
Lemma~\ref{lemma-K-minus} below that $\mathrm{ker}(P_-(2,0)) =
\mathrm{span}\{U\}$ for any value of the parameter
$\mathcal{E} \in (0,1)$, so that the assumption of
Proposition~\ref{proposition-lowest-band} is satisfied.
In this case, the formula (\ref{expansion-lowest-band}) shows that $\mu''(0) > 0$.

Next, we give an explicit expression for $\mu''(0)$ by evaluating the various
terms in (\ref{expansion-lowest-band}) using known properties
of the Jacobi elliptic functions. These computations are performed
in Appendix~A, see equations (\ref{expansion-lowest-band-appendix})--(\ref{norm3-app}),
and yield the explicit formula
\begin{equation}
\label{expression-mu}
\mu''(0) = \frac{2 \ell^2 k^2 (4 k^2 - (c-2)^2 (1+k^2)^2)}{(1+k^2)
\left( 1 - \frac{E(k)}{K(k)}\right) \left(2 k^2 + (c-2) (1+k^2) \left( 1 -  \frac{E(k)}{K(k)}
\right)\right)}\,,
\end{equation}
where $K(k)$ and $E(k)$ are the complete elliptic integrals of the
first and second kind, respectively, and the parameter $k \in (0,1)$
is given by (\ref{exact-c-plus-minus}). The denominator in
(\ref{expression-mu}) is strictly positive if $c \ge 1$. Indeed, since $K(k) > E(k)$ for all $k
\in (0,1)$, thanks to equation (\ref{norm1-app}) in Appendix A, the
denominator in (\ref{expression-mu}) is a strictly increasing function
of $c$, and for $c = 1$ we have
$$
2 k^2 + (c-2)(1+k^2) \left( 1 -  \frac{E(k)}{K(k)} \right) \biggr|_{c = 1}
= k^2-1 + (k^2+1) \frac{E(k)}{K(k)}\,> 0.
$$
The expression above is positive for all $k \in (0,1)$, thanks to
equation (\ref{norm2-app}) in Appendix A.  Thus, for $c \ge 1$, the
sign of $\mu''(0)$ is the sign of the numerator in
(\ref{expression-mu}). It follows that $\mu''(0) \geq 0$ if $c \in
[c_-,c_+] \subset [1,3]$, where $c_{\pm}$ are given by
(\ref{exact-c-plus-minus}). Similarly, we have $\mu''(0) < 0$ if $c
\ge 1$ and $c \notin [c_-,c_+]$.

\begin{remark}
The computations above imply the conclusion of Proposition
\ref{prop-explicit}.  Indeed, either the kernel of $P_-(c,0)$ in
$L^2_{\rm per}(0,2\pi)$ is one-dimensional, in which case the
perturbation argument of Proposition \ref{proposition-lowest-band}
applies and proves the existence of negative spectrum if $c \ge 1$ is
outside $[c_-,c_+]$, or the kernel is higher-dimensional and the
second variation $\Lambda_c''(u_0)$ has more neutral directions than
the two directions due to the symmetries. Note that we do not claim
that the second variation $\Lambda_c''(u_0)$ (or even the
quadratic form associated with $K_-(c)$) is indeed positive if $c \in
(c_-,c_+)$, although by Proposition \ref{proposition-eigenvalues} this
is definitely the case for the periodic waves of small amplitudes.
\end{remark}

\begin{remark}
If we compare the above results with the computations in \cite{Decon},
one advantage of our approach is that we clearly distinguish between
the spectra of the two linear operators $K_+(c)$ and $K_-(c)$. In
particular, the necessary condition in Proposition \ref{prop-explicit}
is derived from the positivity of the Floquet--Bloch spectrum
of $K_-(c)$. We expect that, for any $\mathcal{E} \in (0,1)$,
the Floquet--Bloch spectrum of $K_+(c)$ is positive for $c$ in a larger subset
of $\mathbb{R}$ than $(c_-,c_+)$. For instance, the operator
$K_+(c)$ is positive in $L^2(\mathbb{R})$ for every $c \leq 3$
in the case of the black soliton that corresponds to
$\mathcal{E} = 0$, see Remark \ref{remark-black-plus} below.
\end{remark}

\section{Positive representations of $\Lambda_c''(u_0)$}
\label{sec:positive}

As a first step in the proof of Proposition \ref{prop-positivity},
which claims that the quadratic forms associated with the linear
operators $K_{\pm}(c)$ are nonnegative on $H^2(\R)$ if $c=2$, we
look for representations of these quadratic forms as sums of squared quantities.

Our first result shows
that, if $c = 2$, the quadratic form associated with $K_-(c)$ is
always positive, for all $\mathcal{E} \in [0,1]$, including the black
soliton for $\mathcal{E} = 0$ and the zero solution for
$\mathcal{E} = 1$.

\begin{lemma}
\label{lemma-K-minus}
Fix $c = 2$. For any $\mathcal{E} \in [0,1]$ and any $v \in
H^2(\R)$, we have
\begin{equation}
\label{operator-K-minus}
\langle K_-(2) v, v \rangle_{L^2} = \| L_- v \|_{L^2}^2 +
\| u_0 v_x - u_0' v \|_{L^2}^2.
\end{equation}
\end{lemma}

\begin{proof}
Using the definition (\ref{operatorsdef}) of the operator $L_-$
and integrating by parts, we obtain
\begin{align*}
\| L_- v \|_{L^2}^2 \,&=\, \int_{\R} \left[ v_{xx}^2 + 2 (1 - u_0^2) v v_{xx}
+ (1-u_0^2)^2 v^2 \right] \dd x \\ \,&=\, \int_{\R} \left[ v_{xx}^2
- 2(1 - u_0^2) v_x^2 - 2 (u_0 u_0')' v^2 + (1-u_0^2)^2 v^2 \right] \dd x.
\end{align*}
Similarly, we obtain
$$
 \| u_0 v_x - u_0' v \|_{L^2}^2 = \int_{\R} \left[ u_0^2 v_x^2 +
(u_0 u_0')' v^2 + (u_0')^2 v^2 \right] \dd x.
$$
As a consequence, we have
$$
\| L_- v \|_{L^2}^2 + \| u_0 v_x - u_0' v \|_{L^2}^2 = \int_{\R}
\left[ v_{xx}^2 + (3 u_0^2 - 2) v_x^2 + [(1-u_0^2)^2 - u_0 u_0''] v^2
\right] \dd x,
$$
which yields the desired result since $(1-u_0^2)^2 - u_0 u_0''
= 1 - u_0^2$.
\end{proof}

\begin{remark}
It is easy to verify that the right-hand side of the representation
(\ref{operator-K-minus}) vanishes if and only if $v = C u_0$ for some
constant $C$. As $u_0 \notin H^2(\R)$, this shows that $\langle
K_-(2) v, v \rangle_{L^2} > 0$ for any nonzero $v \in H^2(\R)$.
\end{remark}

Unfortunately, we are not able to find a positive representation
for the quadratic form associated with the operator $K_+(c)$.
If we proceed as in the proof of Lemma~\ref{lemma-K-minus}, we obtain
\begin{equation}
\label{operator-K-plus}
\langle K_+(2) u, u \rangle_{L^2} = \| L_+ u \|_{L^2}^2 - \int_{\R}
\left[ u_0^2 u_x^2 - 3 u_0^2 u^2 + 5 u_0^4 u^2 \right] \dd x.
\end{equation}
Here the second term in the right-hand side has no definite sign,
hence it is difficult to exploit the representation
(\ref{operator-K-plus}). In the following lemma, we give a partial
result which shows that the quadratic form associated with $K_+(c)$ is
positive for $c < 3$ at least on a subspace of $H^2(\R)$.

\begin{lemma}
\label{lemma-K-plus}
For any $\mathcal{E} \in (0,1)$, any $c \in \R$, and any $u\in H^2(\R)$
such that $u(x) = 0$ whenever $u_0'(x) = 0$, we have
\begin{equation}
\label{operator-K-plus-partial}
\langle K_+(c) u, u \rangle_{L^2} = \| w_x \|_{L^2}^2 + (3-c)
\| w \|_{L^2}^2 + 2 \mathcal{E}^2 \left\| \frac{u_0 w}{u_0'}
\right\|_{L^2}^2,
\end{equation}
where $w = u_x - \frac{u_0''}{u_0'} u \in H^1(\R)$ satisfies
$\frac{w}{u_0'} \in L^2(\R)$.
\end{lemma}

\begin{proof}
Since $u_0'$ satisfies the second-order differential equation
$L_+ u_0' = 0$, the zeros of $u_0'$ are all simple, as can also be
deduced from the explicit formula (\ref{sn-periodic}). Thus, if
$u \in H^2(\R)$ is such that $u(x) = 0$ whenever $u_0'(x) = 0$, we
can write $u = u_0'\tilde{u}$ and it follows from Hardy's inequality
that $\tilde{u} \in H^1(\R)$. With this notation, we have
$$
  w := u_x - \frac{u_0''}{u_0'} u = u_x - u_0''\tilde u =
  u_0'\tilde{u}_x,
$$
so that $w \in H^1(\R)$ and $\frac{w}{u_0'} \in L^2(\R)$. As a consequence,
all terms in right-hand side of (\ref{operator-K-plus-partial}) are
well-defined, and the integrations by parts used in the computations
below can easily be justified.

To prove the representation (\ref{operator-K-plus-partial}), we first
note that
$$
u_{xx} + (1 - 3 u_0^2) u = u_{xx} - \frac{u_0'''}{u_0'} u = w_x
+ \frac{u_0''}{u_0'} w.
$$
Integrating by parts, we thus obtain
$$
\| L_+ u \|_{L^2}^2 = \Bigl\| w_x + \frac{u_0''}{u_0'} w
\Bigr\|_{L^2}^2 = \| w_x \|_{L^2}^2 + \int_{\R} \Bigl[
(1 - 3 u_0^2) w^2 + \frac{2 (u_0'')^2}{(u_0')^2} w^2 \Bigr] \dd x.
$$
On the other hand, we have
$$
\| w \|_{L^2}^2 = \int_{\R} \left[ u_x^2 + (3 u_0^2 - 1) u^2 \right] \dd x,
$$
and
$$
\| u_0 w \|_{L^2}^2 = \int_{\R} \left[ u_0^2 u_x^2 + (5 u_0^2 - 3) u_0^2
u^2 \right] \dd x.
$$
Thus, using the analogue of (\ref{operator-K-plus}) for all $c \in
\mathbb{R}$, we find
\begin{align*}
\langle K_+(c) u, u \rangle_{L^2} \,&=\,  \| L_+ u \|_{L^2}^2
+ (2 - c) \int_{\R} \left[ u_x^2 - u^2 + 3 u_0^2 u^2 \right] \dd x
- \int_{\R} \left[ u_0^2 u_x^2 - 3 u_0^2 u^2 + 5 u_0^4 u^2 \right] \dd x \\
\,&=\, \| w_x \|_{L^2}^2 + (3 - c) \| w \|_{L^2}^2 + 2  \int_{\R}
\Bigl[ \frac{(u_0'')^2}{(u_0')^2} - 2 u_0^2 \Bigr] w^2 \dd x,
\end{align*}
which yields the desired result since $\frac{(u_0'')^2}{(u_0')^2}
- 2 u_0^2 = \mathcal{E}^2 \frac{u_0^2}{(u_0')^2}$ holds by equations
(\ref{wave}) and (\ref{first-order}).
\end{proof}

\begin{remark}
If $c \le 3$, the right-hand side of the representation
(\ref{operator-K-plus-partial}) is nonnegative and vanishes if and
only if $w = 0$, which is equivalent to $u = C u_0'$ for some constant
$C$. However, this does not imply positivity of the quadratic form
associated to $K_+(c)$, because the representation
(\ref{operator-K-plus-partial}) only holds for $u$ in a subspace of
$H^2(\R)$. As a matter of fact, the right-hand side of the
representation (\ref{operator-K-plus-partial}) is positive for any $c
\le 3$, whereas we know from the proof of
Proposition~\ref{proposition-eigenvalues} that, when $\mathcal{E}$ is
close to $1$, the operator $K_+(c)$ is positive if and only if $c \in
(c_-,c_+)$ where $c_{\pm} \to 2$ as $\mathcal{E} \to 1$.
\end{remark}

For the black soliton (\ref{black-soliton}) corresponding to the case
$\mathcal{E} = 0$, the proof of Lemma~\ref{lemma-K-plus} yields a much
stronger conclusion, because $u_0'$ never vanishes so that we do not
need to impose any restriction to $u \in H^2(\R)$. Using the identity
$u_0'' = -\sqrt{2}u_0 u_0'$ which holds for the black soliton
(\ref{black-soliton}) only, we obtain the following result.

\begin{corollary}
\label{corollary-K-plus}
Consider the black soliton (\ref{black-soliton}), for which
$\mathcal{E} = 0$. For any $c \in \R$ and any $u \in H^2(\R)$,
we have
\begin{equation}
\label{operator-K-plus-soliton}
\langle K_+(c) u, u \rangle_{L^2} = \| w_x \|_{L^2}^2 + (3-c)
\| w \|_{L^2}^2,
\end{equation}
where $w = u_x + \sqrt{2} u_0 u \in H^1(\R)$.
\end{corollary}

\begin{remark}
If $c \le 3$, the right-hand side of the representation
(\ref{operator-K-plus-soliton}) is nonnegative and vanishes if and
only if $w = 0$, which is equivalent to $u = C u_0'$ for some constant
$C$. Note that $u_0' \in H^2(\R)$ in the present case. On the other
hand, using definitions (\ref{operatorsdef}) and the fact that $u_0(x)
\to \pm 1$ as $x \to \pm \infty$, it is easy to verify that $K_+(c)$
has some negative essential spectrum as soon as $c > 3$. Thus the
representation (\ref{operator-K-plus-soliton}) gives a sharp
positivity criterion for the operator $K_+(c)$ in the case of the black soliton
(\ref{black-soliton}).
\label{remark-black-plus}
\end{remark}

\section{Positivity of $\Lambda_{c=2}''(u_0)$ for periodic waves
of large amplitude}
\label{sec:operator}

This section presents the proof of Proposition \ref{prop-positivity}.

The energy functionals (\ref{energy}) and (\ref{energy-S}) generate
two different flows in the hierarchy of integrable NLS equations, see
\cite{Decon}. If we consider $E$ and $S$ as functions of the complex
variables $\psi$ and $\bar{\psi}$, these flows are defined by the
evolution equations
\begin{equation}\label{2flows}
i \frac{\partial \psi}{\partial t} = \frac{\delta E}{\delta \bar{\psi}},
\qquad i \frac{\partial \psi}{\partial \tau} = \frac{\delta S}{\delta
\bar{\psi}},
\end{equation}
where the symbol $\delta$ is used to denote the standard variational
derivative. Here $t$ is the time of the cubic defocusing NLS equation
(\ref{nls}), whereas $\tau$ is the time of the higher-order NLS
equation. Since the quantities $E$ and $S$ are in involution, the
flows defined by both equations in (\ref{2flows}) commute with each
other.

In what follows, we fix some $\mathcal{E} \in (0,1)$ and consider the
periodic wave profile $u_0$ defined by (\ref{sn-periodic}).  Using the
real-valued variables $u,v$ for the perturbations, as in the
representations (\ref{secondE}) and (\ref{secondS}), we obtain the
following evolution equations for the linearized flows of the cubic
NLS equation and the higher-order NLS equation at the periodic wave
profile $u_0$:
\begin{equation}
\label{lin-operators-flow}
\frac{\partial}{\partial t} \left[ \begin{array}{c} u \\ v \end{array} \right]
= \left[ \begin{array}{cc} 0 & L_- \\ -L_+ & 0 \end{array} \right]
\left[ \begin{array}{c} u \\ v \end{array} \right], \qquad
\frac{\partial}{\partial \tau} \left[ \begin{array}{c} u \\ v \end{array}
\right] = \left[ \begin{array}{cc} 0 & M_- \\ -M_+ & 0 \end{array} \right]
\left[ \begin{array}{c} u \\ v \end{array} \right],
\end{equation}
where the operators $L_{\pm}$ and $M_{\pm}$ are given by
(\ref{operatorsdef}). Because the linearized flows also commute with
each other, the operators $L_{\pm}$ and $M_{\pm}$ satisfy the
following intertwining relations
\begin{equation}
\label{intertwining}
L_- M_+ = M_- L_+, \qquad L_+ M_- = M_+ L_-.
\end{equation}
Of course, the relations (\ref{intertwining}) can also be verified by
a direct calculation, using the differential equations (\ref{wave})
and (\ref{first-order}) satisfied by the periodic wave profile
$u_0$. It follows from the relation (\ref{intertwining}) that, for
every $c \in \mathbb{R}$, we have
\begin{equation}
\label{intertwiningK}
L_- K_+(c) = K_-(c) L_+, \qquad L_+ K_-(c) = K_+(c) L_-,
\end{equation}
where $K_{\pm}(c) = M_{\pm} - c L_{\pm}$ as before.

Given the positivity of the operator $K_-(2)$ established in Lemma
\ref{lemma-K-minus}, we shall use the intertwining relations
(\ref{intertwiningK}) to deduce the positivity of the operator
$K_+(2)$. This is achieved by studying all bounded solutions of the
homogeneous equations associated with operators $L_{\pm}$ and
$K_{\pm}(2)$ and by applying a continuation argument from the limit
$\mathcal{E} \to 1$, where positivity of the operator $K_+(2)$ is
proved in Proposition \ref{proposition-eigenvalues}.

\begin{lemma}
\label{lemma-L-plus}
If $u \in L^{\infty}(\mathbb{R}) \cap H^2_{\rm loc}(\mathbb{R})$
satisfies $L_+ u = 0$, then $u = C u_0'$ for some constant $C$.
Moreover, there exists a unique odd, $2T_0$-periodic function $U \in
H^2_{\rm per, odd}(0,2T_0)$ such that $L_+ U = u_0$, where $2T_0$ is
the period of $u_0$.
\end{lemma}

\begin{proof}
We know that $L_+ u_0' = 0$. Another linearly independent solution to
the equation $L_+ v = 0$ can be obtained by differentiating the
periodic wave profile $u_0$ with respect to the parameter $\mathcal{E}
\in (0,1)$, namely $v = \partial_\mathcal{E} u_0$.  Indeed, if we
differentiate the equation
$$
u_0'' + u_0 - u_0^3 = 0
$$
with respect to the parameter $\mathcal{E}$, we see that
$$
L_+ v = -v'' + (3u_0^2-1) v = 0.
$$
Moreover, $v(x)$ is an odd function of $x$ that grows linearly as $|x| \to
\infty$. The latter claim can be verified by differentiating the
explicit formula (\ref{sn-periodic}) with respect to $\mathcal{E}$,
but that calculation is not immediate because it involves the
derivative of the Jacobi elliptic function ${\rm sn}(\xi,k)$ with
respect to the parameter $k$. Alternatively, we can use Floquet theory
to deduce that $v$ is either periodic of period $2T_0$, where $2T_0$
is the minimal period of $u_0$, or grows linearly at infinity. The
second possibility is excluded by the following argument.  If we
denote $u_0(x) = u_0(x;\mathcal{E})$ and $T_0 = T_0(\mathcal{E})$ to
emphasize the dependence upon the parameter $\mathcal{E}$, we have
by contruction
$$
  u_0(0;\mathcal{E}) = u_0(2T_0(\mathcal{E});\mathcal{E}) = 0.
$$
Differentiating that relation with respect to $\mathcal{E}$, we find
$$
v(0) = 0 \quad \mbox{\rm and}\quad
v(2T_0) + 2u_0'(2T_0) T_0'(\mathcal{E}) = 0.
$$
But we
know that $u_0'(2T_0) = u_0'(0) > 0$ and that $T_0'(\mathcal{E}) < 0$,
hence we deduce that $v(2T_0) > 0$, which implies that $v$ is not
periodic of period $2T_0$. This proves that the kernel of $L_+$ (in
the space of bounded functions) is spanned by $u_0'$, which is the
first part of the statement.

For the second part of the statement, we look for solutions of the
inhomogeneous equation $L_+ U = u_0$ and note that the Fredholm
solvability condition $\langle u_0',u_0 \rangle_{L^2_{\rm per}} = 0$
is trivially satisfied in the space of $2T_0$-periodic
functions. Hence, there exists a unique odd $2T_0$-periodic solution
$U$ of the inhomogeneous equation $L_+ U = u_0$ in the domain of
$L_+$, that is, $U \in H^2_{\rm per, odd}(0,2T_0)$.
\end{proof}

\begin{lemma}
\label{lemma-L-minus}
If $v \in L^{\infty}(\mathbb{R}) \cap H^2_{\rm loc}(\mathbb{R})$
satisfies $L_- v = 0$, then $v = C u_0$ for some constant $C$.
Moreover, there exists a unique even, $2T_0$-periodic function $V \in
H^2_{\rm per, even}(0,2T_0)$ such that $L_- V = u_0'$, where $2T_0$ is
the period of $u_0$.
\end{lemma}

\begin{proof}
We know that $L_- u_0 = 0$. Another linearly independent solution
to the equation $L_- v = 0$ is given by
$$
  v(x) = 2u_0'(x) - u_0(x) \int_0^x u_0(y)^2 \dd y, \qquad
  x \in \R,
$$
as is easily verified by a direct calculation. Clearly $v(x)$ is
an even function of $x$ that grows linearly as $|x| \to \infty$.
This proves that the kernel of $L_-$ (in the space of bounded functions)
is spanned by $u_0$. The second part of the statement follows by
the same argument as in the proof of Lemma~\ref{lemma-L-plus}.
\end{proof}

\begin{remark}
The solutions $U$ and $V$ of the inhomogeneous equations $L_+ U = u_0$
and $L_- V = u_0'$ can be expressed explicitly in terms of the Jacobi
elliptic functions, see equations (\ref{L0ident}) and (\ref{LPlusUU})
in Appendix A.
\end{remark}

Next, we establish analogues of Lemmas \ref{lemma-L-plus} and
\ref{lemma-L-minus} for the operators $K_{\pm}(c)$ in the particular
case $c = 2$.

\begin{lemma}
\label{lemma-K-minus-positive}
If $v \in L^{\infty}(\mathbb{R}) \cap H^4_{\rm loc}(\mathbb{R})$
satisfies $K_-(2) v = 0$, then $v = C u_0$ for some constant $C$.
\end{lemma}

\begin{proof}
Using integration by parts as in the proof of Lemma~\ref{lemma-K-minus},
we obtain the following identity for any $v \in H^4(-NT_0,NT_0)$, where
$N \in \mathbb{N}$ and $2T_0$ is the period of $u_0$:
$$
\int_{-NT_0}^{NT_0} v K_-(2) v\dd x = \int_{-NT_0}^{NT_0} \left( |L_- v|^2 +
|u_0 v_x - u_0' v|^2 \right) dx - \Bigl[ 2 (1-u_0^2) v v_x + u_0 u_0'
v^2 \Bigr] \Big|_{x = -NT_0}^{x = N T_0}.
$$
Assume now that $v \in L^{\infty}(\mathbb{R}) \cap H^4_{\rm loc}
(\mathbb{R})$ satisfies $K_-(2) v = 0$. By standard elliptic
estimates, we know that $v$ is smooth on $\mathbb{R}$ and that all
derivatives of $v$ are bounded. Moreover, since the operator $K_-(2)$
has $T_0$-periodic coefficients, it follows from Floquet theory that
$v(x) = e^{i \gamma x} w(x)$, where $\gamma \in \mathbb{R}$ and $w$ is
smooth on $\mathbb{R}$ and $T_0$-periodic. Using the identity above,
we thus obtain
\begin{align*}
0 \,&=\, \frac{1}{N} \int_{-NT_0}^{NT_0} \left( |L_- v|^2 +
|u_0 v_x - u_0' v|^2 \right) \dd x - \frac{1}{N}\Bigl[ 2 (1-u_0^2)
v v_x + u_0 u_0' v^2 \Bigr] \Big|_{x = -NT_0}^{x = N T_0} \\
\,&=\, \int_{-T_0}^{T_0} \left( |L_- v|^2 + |u_0 v_x - u_0' v|^2 \right)
\dd x - \frac{1}{N}\Bigl[ 2 (1-u_0^2) v v_x + u_0 u_0' v^2 \Bigr]
\Big|_{x = -NT_0}^{x = N T_0}.
\end{align*}
Taking the limit $N \to \infty$ and using the boundedness of $v$ and
$v_x$, we obtain $L_- v = 0$ and $u_0 v_x - u_0' v = 0$ for all $x \in
\mathbb{R}$ (since $v(x) = e^{i \gamma x} w(x)$ and $w$ is
$T_0$-periodic). By Lemma \ref{lemma-L-minus}, we conclude that $v = C
u_0$ for some constant $C$.
\end{proof}

\begin{lemma}
\label{lemma-K-plus-positive}
If $u \in L^{\infty}(\mathbb{R}) \cap H^4_{\rm loc}(\mathbb{R})$
satisfies $K_+(2) u = 0$, then $u = C u_0'$ for some constant $C$.
\end{lemma}

\begin{proof}
Assume that $u \in L^{\infty}(\mathbb{R}) \cap H^4_{\rm loc}(\mathbb{R})$
satisfies $K_+(2) u = 0$.  By the intertwining
relation (\ref{intertwiningK}), we have $K_-(2) L_+ u = L_- K_+(2) u =
0$.  Using Lemma~\ref{lemma-K-minus-positive}, we deduce that $L_+ u =
B u_0$ for some constant $B$. Finally, Lemma~\ref{lemma-L-plus} implies
that $u = B U + C u_0'$ for some constant $C$. In particular, we have
$0 = K_+(2) u = B K_+(2) U$, because  $K_+(2) u_0' = 0$.
Now an explicit computation that is carried out in Appendix~A shows
that $K_+(2) U = D u_0$ for some constant $D \neq 0$, see equation 
(\ref{MLPlut}), so that $K_+(2)
U$ is not identically zero. Thus $B = 0$, hence $u = C u_0'$.
\end{proof}

\begin{remark}
The result of Lemma \ref{lemma-K-plus-positive} yields the conclusion
of Proposition \ref{prop-positivity}. Indeed, in the limit
$\mathcal{E} \to 1$, positivity of the operator $K_+(2)$ is proved in
Proposition \ref{proposition-eigenvalues}.  All Floquet--Bloch bands
are strictly positive, except for the lowest band that touches the
origin because of the zero eigenvalue due to translational
symmetry, see Figure~\ref{fig2}.  When the
parameter $\mathcal{E}$ is decreased from $1$ to $0$, the
Floquet--Bloch spectrum of $K_+(2)$ evolves continuously, and
positivity of the spectrum is therefore preserved as long as no other
band touches the origin. Such an event would result in the
appearance of another bounded solution to the homogeneous equation
$K_+(2) u = 0$, besides the zero mode $u_0'$ due to translation
invariance. By Lemma~\ref{lemma-K-plus-positive}, such a solution
does not exist, hence $K_+(2)$ is a nonnegative operator for
any $\mathcal{E} \in (0,1)$.
\end{remark}

To conclude this section, we note that the intertwining relations
(\ref{intertwiningK}) and the positivity of the operators $K_{\pm}(2)$
established in Proposition~\ref{prop-positivity} imply the
spectral stability of the periodic wave. Consider the
linearized operator with $T_0$-periodic coefficients given by
\begin{equation}
\label{operator-lin}
\mathcal{J} \mathcal{L} \,:=\,  \left[ \begin{array}{cc} 0 & 1 \\ -1 & 0
\end{array} \right] \left[ \begin{array}{cc} L_+ & 0 \\ 0 & L_-
\end{array} \right] \,=\, \left[ \begin{array}{cc} 0 & L_- \\ -L_+ & 0
\end{array} \right],
\end{equation}
and acting on vectors in $L^2(\mathbb{R}) \times L^2(\mathbb{R})$. We say that the periodic wave is
spectrally stable if the Floquet--Bloch spectrum of $\mathcal{J}
\mathcal{L}$ is purely imaginary. Let $\lambda \in \mathbb{C}$ belong
to the Floquet--Bloch spectrum, so that $\mathcal{J}\mathcal{L} \psi =
\lambda\psi$ for some nonzero eigenfunction $\psi$. We know that
$\psi(x) = e^{i \gamma x} \tilde{\psi}(x)$, where $\gamma \in
\mathbb{R}$ and $\tilde{\psi}$ is $T_0$-periodic. We want to show that
$\lambda \in i \mathbb{R}$.

Let $\mathcal{K} := {\rm diag}[K_+(2),K_-(2)]$. Then
$\mathcal{J}\mathcal{L}\mathcal{J} \mathcal{K}\psi =
\mathcal{J} \mathcal{K}\mathcal{J} \mathcal{L} \psi =
\lambda \mathcal{J} \mathcal{K} \psi$, because the operators
$\mathcal{J} \mathcal{L}$ and $\mathcal{J}\mathcal{K}$
commute due to the intertwining relations (\ref{intertwiningK}).
As $\mathcal{J}$ is invertible, we thus have $\mathcal{L}
\mathcal{J} \mathcal{K} \psi = \lambda \mathcal{K}\psi$.
If we now take the scalar product of both sides with the eigenfunction $\psi$ in
the space $L^2(0,T_0) \times L^2(0,T_0)$, we obtain
$$
\lambda \langle \psi, \mathcal{K} \psi \rangle_{L^2} = \langle \psi,
\mathcal{L} \mathcal{J} \mathcal{K} \psi \rangle_{L^2} =
- \langle \mathcal{J} \mathcal{L} \psi, \mathcal{K} \psi \rangle_{L^2}
= - \bar{\lambda} \langle \psi, \mathcal{K} \psi \rangle_{L^2},
$$
where we have used the fact that $\mathcal{L}$ is self-adjoint and
$\mathcal{J}$ is skew-adjoint. If $\lambda \neq 0$, then $\psi$ is
not a linear combination of the two neutral eigenfunctions
$(u_0',0)$ and $(0,u_0)$. In that case, we have $\langle
\psi, \mathcal{K} \psi \rangle_{L^2} > 0$ by
Proposition~\ref{prop-positivity}, and the identity above shows
that $\lambda = -\bar{\lambda}$, that is, $\lambda \in i \mathbb{R}$.

\begin{remark}
Spectral stability of the periodic wave is established in \cite{Decon},
where explicit expressions for the Floquet--Bloch spectrum of the operator
$\mathcal{J} \mathcal{L}$ and the associated eigenfunctions are
obtained using Jacobi elliptic functions. In our approach, once
positivity of the operator $\mathcal{K}$ is known, the spectral
stability of the periodic wave follows from the commutativity of the
operators $\mathcal{J} \mathcal{L}$ and $\mathcal{J} \mathcal{K}$ and
is established by a general argument that does not use the specific
form of the eigenfunctions.
\end{remark}

\section{Proof of orbital stability of a periodic wave}
\label{sec:periodic}

This section is devoted to the proof of Theorem~\ref{theorem-per}.

We fix $\mathcal{E}\in (0,1)$ and consider the periodic wave profile
$u_0$ given by (\ref{sn-periodic}).  Let $T$ be a multiple of the
period $2T_0$ of $u_0$, so that $T = 2NT_0$ for some integer $N \ge 1$. If
$\psi_0 \in H^2_{\rm per}(0,T)$ is close to $u_0$ in the sense of the
initial bound (\ref{bound-initial-per}), we claim that the solution
$\psi \in C(\R,H^2_{\rm per}(0,T))$ of the cubic NLS equation
(\ref{nls}) with initial data $\psi_0$ can be characterized as
follows.

For any $t \in \R$, there exist modulation parameters $\xi(t) \in \R$
and $\theta(t) \in \R/(2\pi\Z)$ such that
\begin{equation}
\label{decomposition}
e^{it+i \theta(t)} \psi(x + \xi(t),t) = u_0(x) + u(x,t) + i v(x,t),
\quad x \in \R,
\end{equation}
where $u(\cdot,t), v(\cdot,t) \in H^2_{\rm per}(0,T)$ are real-valued
functions satisfying the orthogonality conditions
\begin{equation}
\label{projections}
\langle u_0', u(\cdot,t) \rangle_{L^2_{\rm per}} = 0, \qquad \langle
u_0, v(\cdot,t) \rangle_{L^2_{\rm per}} = 0,
\end{equation}
where $\langle \cdot\,, \cdot\rangle_{L^2_{\rm per}}$ denotes the
usual scalar product in $L^2_{\rm per}(0,T)$. Note that the
orthogonality conditions (\ref{projections}) are not symplectic
orthogonality conditions for the NLS equation, in contrast with the
conditions that are often used to study the asymptotic stability of
nonlinear waves \cite{PelBook}.

To prove the decomposition (\ref{decomposition}), we proceed in two
steps. We first show that the representation (\ref{decomposition})
holds whenever $\psi(\cdot,t)$ is sufficiently close to the orbit of
$u_0$ under translations and phase rotations.

\begin{lemma}
\label{lemma-a-theta-per1}
There exists constants $\epsilon_0 \in (0,1)$ and $C_0 \ge 1$ such that,
for any $\psi \in H^2_{\rm per}(0,T)$ satisfying
\begin{equation}
\label{infinimum}
d := \inf_{\xi, \theta \in \R} \| e^{i \theta} \psi(\cdot + \xi)
-  u_0 \|_{H^2_{\rm per}} \le \epsilon_0,
\end{equation}
one can find modulation parameters $\xi \in \R$ and $\theta \in \R/(2\pi\Z)$ such that
\begin{equation}
\label{decomposition-lemma}
e^{i \theta} \psi(x + \xi) = u_0(x) + u(x) + i v(x), \quad x \in \R,
\end{equation}
where $u,v \in H^2_{\rm per}(0,T)$ satisfy the orthogonality conditions
(\ref{projections}) and $d \le \|u + i v\|_{H^2_{\rm per}} \le C_0d$.
\end{lemma}

\begin{proof}
We consider the smooth function ${\bf f} : \R^2 \to \R^2$ defined by
$$
{\bf f}(\xi,\theta) = \left[ \begin{array}{l} \langle u_0'(\cdot
- \xi), {\rm Re}(e^{i \theta} \psi) \rangle_{L^2_{\rm per}} \\
 \langle u_0(\cdot - \xi), {\rm Im}(e^{i \theta} \psi)
\rangle_{L^2_{\rm per}} \end{array} \right].
$$
We have ${\bf f}(\xi,\theta) = {\bf 0}$ if and only if $\psi$
can be represented as in the decomposition (\ref{decomposition-lemma}) with
$u,v$ satisfying the orthogonality conditions (\ref{projections}).
Let $(\xi_0,\theta_0) \in \R^2$ denote the arguments of the infinimum
in (\ref{infinimum}) (note that one can restrict the values of
$(\xi,\theta)$ to $[0,T] \times [0,2\pi]$, so that the minimum
exists). Then assumption (\ref{infinimum}) implies that $\|{\bf f}
(\xi_0,\theta_0) \| \le C d$, for some constant $C$
independent of $\psi$. On the other hand, the Jacobian matrix of
the function ${\bf f}$ at the point $(\xi_0,\theta_0)$ is given by
\begin{align*}
D {\bf f}(\xi_0,\theta_0) \,&=\,
\left[ \begin{array}{cc} \| u_0' \|_{L^2_{\rm per}}^2 & 0 \\ 0 &
\| u_0 \|_{L^2_{\rm per}}^2 \end{array} \right] \\
&\quad\, + \left[ \begin{array}{cc} -\langle u_0'',
{\rm Re}(e^{i \theta_0} \psi(\cdot + \xi_0) - u_0) \rangle_{L^2_{\rm per}} &
-\langle u_0', {\rm Im}(e^{i \theta_0} \psi(\cdot + \xi_0) - u_0)
\rangle_{L^2_{\rm per}} \\
-\langle u_0', {\rm Im}(e^{i \theta_0} \psi(\cdot + \xi_0) - u_0)
\rangle_{L^2_{\rm per}} &
\langle u_0, {\rm Re}(e^{i \theta_0} \psi(\cdot + \xi_0) - u_0)
\rangle_{L^2_{\rm per}} \end{array} \right].
\end{align*}
The first term in the right-hand side is a fixed invertible matrix
and the second term is bounded in norm by $Cd$, hence $D {\bf f}
(\xi_0,\theta_0)$ is invertible if $\epsilon_0$ is small enough,
with $\|(D {\bf f}(\xi_0,\theta_0))^{-1}\| \le C$ where $C$ is a
positive constant independent of $\psi$. Finally, it is
straightforward to verify that the second order derivative of
${\bf f}$ is uniformly bounded if $\epsilon_0 < 1$. These observations
together imply that there exists a unique pair $(\xi,\theta)$, in the
$\mathcal{O}(d)$ neighborhood of the point $(\xi_0,\theta_0)$, such
that ${\bf f}(\xi,\theta) = {\bf 0}$.  Thus, we have the decomposition
(\ref{decomposition-lemma}) with these values of $(\xi,\theta)$, and
\begin{align*}
  \|u+iv\|_{H^2_{\rm per}} \,&=\, \|e^{i\theta}\psi(\cdot+\xi)-u_0\|_{H^2_{\rm per}}
  = \|\psi - e^{-i\theta}u_0(\cdot-\xi)\|_{H^2_{\rm per}}\\
  \,&\le\, \|\psi - e^{-i\theta_0}u_0(\cdot-\xi_0)\|_{H^2_{\rm per}} +
  \|e^{-i\theta_0}u_0(\cdot-\xi_0) - e^{-i\theta}u_0(\cdot-\xi)\|_{H^2_{\rm per}} \\
  \,&\le\, C_0d,
\end{align*}
where $C_0 \geq 1$ is independent of $\psi$. This concludes the proof.
\end{proof}

We next show that the solution $\psi(\cdot,t)$ of the cubic NLS
equation (\ref{nls}) stays close to the orbit of $u_0$ for all
times. To show this, we use the conserved quantity $\Lambda_c$ given
by (\ref{Lyapunov-functional}), where it is understood that the
integration domain $\I = (0,T)$ is used in the definitions of all
functionals (\ref{energy}), (\ref{charge}), and (\ref{energy-R}).
Because positivity of the second variation of $\Lambda_c$ is only proved
for $c = 2$ independently of the parameter $\mathcal{E}$, see
Proposition \ref{prop-positivity}, we assume henceforth that $c = 2$.

\begin{lemma}
\label{Lambdalem}
Assume that $\psi$ is given by (\ref{decomposition-lemma})
for some $(\xi,\theta) \in \R^2$ and some real-valued functions
$u, v \in H^2_{\rm per}(0,T)$ satisfying the orthogonality conditions
(\ref{projections}). There exist positive constants $C_1$, $C_2$,
and $\epsilon_1$ such that, if $\|u + i v\|_{H^2_{\rm per}} \le \epsilon_1$, then
\begin{equation}
\label{Lambdacomp}
C_1 \|u + i v\|_{H^2_{\rm per}}^2  \le  \Lambda_{c=2}(\psi) - \Lambda_{c=2}(u_0)
\le C_2 \|u + i v \|_{H^2_{\rm per}}^2.
\end{equation}
\end{lemma}

\begin{proof}
We first note that the functional $\Lambda_c$ is invariant under
translations and phase rotations in $H^2_{\rm per}(0,T)$, so that
$\Lambda_c(\psi) = \Lambda_c(u_0 + u + iv)$ if $\psi$ satisfies
the representation (\ref{decomposition-lemma}). Therefore, recalling
that $u_0$ is a critical point of $\Lambda_c$ and using the same
notations as in Section~\ref{sec:small}, we find
\begin{equation}
\label{DeltaLambda}
\Lambda_c(\psi) - \Lambda_c(u_0) = \langle K_+(c) u, u
\rangle_{L^2_{\rm per}} + \langle K_-(c) v, v \rangle_{L^2_{\rm per}}
+ N_c(u,v),
\end{equation}
where $N_c(u,v)$ collects all terms that are at least cubic
in $(u,v)$. In particular, there exists a constant $C > 0$ such
that, if $\|u + i v \|_{H^2_{\rm per}} \le \epsilon_1$, we have the estimate
\begin{equation}
\label{bound-cubic-terms}
|N_c(u,v)| \le C  \|u + i v\|^3_{H^2_{\rm per}}.
\end{equation}
The upper bound in (\ref{Lambdacomp}) holds from the expressions
(\ref{secondE}) and (\ref{secondS}) for the quadratic part,
the estimate (\ref{bound-cubic-terms}) for the cubic and quartic parts,
and the decomposition (\ref{DeltaLambda}).

To bound the expression (\ref{DeltaLambda}) from below, we use the
spectral properties of the operators $K_\pm(c)$ established in
Sections~\ref{sec:small}, \ref{sec:positive}, and \ref{sec:operator}. 

For periodic waves of small amplitude and for $c$ in the interval $(c_-,c_+)$, 
we know from Propositions~\ref{prop-per} and \ref{proposition-eigenvalues} that the spectrum of
$K_\pm(c)$ in $L^2(\R)$ is the union of the nonnegative Floquet--Bloch spectral bands. 
If $K_\pm(c)$ are considered as operators in $L^2_{\rm per}(0,T)$ with $T = 2NT_0$, the same result
holds except that the Floquet parameter only takes discrete values. In view of the bounds
(\ref{expansion-lambda-bounds}), this discretization of the Floquet--Bloch spectral bands 
implies that both $K_+(c)$ and $K_-(c)$ have exactly one zero eigenvalue, and that the rest of the
spectrum is positive and bounded away from zero. As was already
observed, the kernels of $K_\pm(c)$ are due to the symmetries of the
NLS equation, and we have the explicit formulas
(\ref{zero-eigenfunctions}) for the eigenvectors. Thus, the orthogonality conditions
(\ref{projections}) mean precisely that $u$ is orthogonal in $L^2_{\rm per}(0,T)$ 
to the kernel of $K_+(c)$ and $v$ to the kernel of $K_-(c)$. 

Although the results of Propositions \ref{prop-per} and \ref{proposition-eigenvalues}
holds for periodic waves of small amplitude where $\mathcal{E}$ is close to one,
Proposition \ref{prop-positivity} implies that the same result hold for periodic waves of 
arbitrary amplitude independently of the parameter $\mathcal{E} \in (0,1)$ in the case $c = 2$.
It then follows that there is a positive constant $C$
such that
$$
\langle K_+(2) u, u \rangle_{L^2_{\rm per}} \ge C \|u\|_{L^2_{\rm per}}^2 \quad
\mbox{\rm and} \quad
\langle K_-(2) v, v \rangle_{L^2_{\rm per}} \ge C \|v\|_{L^2_{\rm per}}^2.
$$
Using in addition G{\aa}rding's inequality for the elliptic operators
$K_\pm(c)$ we conclude that
\begin{equation}
\label{bound-quadratic-forms}
\langle K_+(2) u, u \rangle_{L^2_{\rm per}} \ge C \|u\|_{H^2_{\rm per}}^2, \quad
\langle K_-(2) v, v \rangle_{L^2_{\rm per}} \ge C \|v\|_{H^2_{\rm per}}^2,
\end{equation}
with a possibly smaller constant $C$. The lower bound in
(\ref{Lambdacomp}) is a direct consequence of (\ref{DeltaLambda}),
(\ref{bound-cubic-terms}), and (\ref{bound-quadratic-forms}).
\end{proof}

Without loss of generality, we assume from now on that $C_0 \epsilon_0
\le \epsilon_1$, where $C_0$, $\epsilon_0$, and $\epsilon_1$ are as in
the previous lemmas. It then follows from
Lemmas~\ref{lemma-a-theta-per1} and \ref{Lambdalem} that, if $\psi \in
H^2_{\rm per}(0,T)$ is close to the orbit of $u_0$ in the sense of the
bound (\ref{infinimum}), then
\begin{equation}
\label{Lamequiv}
C_1 d^2 \le \Lambda_{c=2}(\psi) - \Lambda_{c=2}(u_0) \le C_2 C_0^2 d^2.
\end{equation}
With this estimate at hand, it is now easy to prove that the
decomposition (\ref{decomposition}) with the orthogonality conditions
(\ref{projections}) hold for all $t \in \R$ if $\psi(\cdot,t)$ is the
solution of the cubic NLS equation (\ref{nls}) with initial data
$\psi_0 \in H^2_{\rm per}(0,T)$ satisfying the initial bound
(\ref{bound-initial-per}), where $\delta > 0$ is small enough so that
\begin{equation}
\label{deltasmall}
  C_0 (C_2/C_1)^{1/2} \delta < \epsilon_0.
\end{equation}
Indeed, let $d(t)$ be the distance in $H^2_{\rm per}(0,T)$ from
$\psi(\cdot,t)$ to the orbit of $u_0$, in the sense of
(\ref{infinimum}). Initially we have $d(0) \le \delta < \epsilon_0$ by
(\ref{bound-initial-per}) and (\ref{deltasmall}). Let $\mathcal{J}
\subset \R$ be the largest time interval containing the origin such
that $d(t) \le \epsilon_0$ for all $t \in \mathcal{J}$. As $d(t)$ is a
continuous function of time, it is clear that $\mathcal{J}$ is
closed. On the other hand, for any $t \in \mathcal{J}$, we have by
(\ref{Lamequiv})
$$
  C_1 d(t)^2 \le \Lambda_{c=2}(\psi(\cdot,t)) - \Lambda_{c=2}(u_0)
  = \Lambda_{c=2}(\psi_0) - \Lambda_{c=2}(u_0) \le C_2 C_0^2 \delta^2,
$$
where we have used the crucial fact that $\Lambda_c$ is conserved
under the evolution defined by the cubic NLS equation (\ref{nls}) in
$H^2_{\rm per}(0,T)$.  Thus $d(t) \le C_0 (C_2/C_1)^{1/2} \delta <
\epsilon_0$, hence by continuity the interval $\mathcal{J}$ contains
a neighborhood of $t$. So $\mathcal{J}$ is open, hence finally
$\mathcal{J} = \R$.  This shows that the decomposition
(\ref{decomposition}) holds for all $t \in \R$ with real-valued
functions $u(\cdot,t), v(\cdot,t) \in H^2_{\rm per}(0,T)$ satisfying
the orthogonality conditions (\ref{projections}) as well as the
uniform bound
$$
  \|u(\cdot,t) + i v (\cdot,t)\|_{H^2_{\rm per}}
  \le C_0 d(t) \le C_0^2 (C_2/C_1)^{1/2} \delta, \quad t \in \R.
$$
This yields the bound (\ref{bound-final-per}) with $\epsilon = C_0^2
(C_2/C_1)^{1/2} \delta$.  To conclude the proof of
Theorem~\ref{theorem-per}, it remains to show that the modulation
parameters $\xi$ and $\theta$ are continuously differentiable
functions of time $t$ and satisfy the bound (\ref{bound-time-per}).

\begin{lemma}
\label{lemma-a-theta-per2}
Assume that the solution $\psi(\cdot,t)$ of the cubic NLS equation
(\ref{nls}) satisfies $d(t) \le \epsilon \le \epsilon_1$ for all $t
\in \R$, where $d(t)$ denotes as in (\ref{infinimum}) the distance to
the orbit of $u_0$. Then the modulation parameters $\xi(t),\theta(t)$
given by Lemma~\ref{lemma-a-theta-per1} are 
continuously differentiable functions of $t$ satisfying
(\ref{bound-time-per}).
\end{lemma}

\begin{proof}
As $\psi \in C(\R,H^2_{\rm per}(0,T))$, the proof of
Lemma~\ref{lemma-a-theta-per1} shows that $\xi(t)$ and $\theta(t)$
depend continuously on $t$. To
prove differentiability, we first consider more regular solutions
with initial data $\psi_0 \in H^4_{\rm per}(0,T)$, and then recover
the general case by a density argument. For regular solutions,
we can differentiate both sides of the decomposition (\ref{decomposition})
and use the cubic NLS equation (\ref{nls}) to obtain the evolution system
$$
\left\{ \begin{array}{l}
~\,\,u_t = L_- v + \dot{\xi} (u_0' + u_x) - \dot{\theta} v +
(2 u_0 u + u^2 + v^2) v, \\
-v_t = L_+ u -\dot{\xi} v_x - \dot{\theta} (u_0 + u) +
(3 u_0 u + u^2 + v^2) u + u_0 v^2, \end{array} \right.
$$
where the operators $L_{\pm}$ are defined in (\ref{operatorsdef}).
Using the orthogonality conditions (\ref{projections}), we
eliminate the time derivatives $u_t, v_t$ by taking the scalar
product of the first line with $u_0'$ and of the second line with
$u_0$. This gives the following linear system for the derivatives
$\dot{\xi}$ and $\dot{\theta}$:
\begin{equation}
\label{Bsys}
B \left[ \begin{array}{c} \dot{\xi} \\ \dot{\theta} \end{array} \right]
= \left[ \begin{array}{c} \langle u_0', L_- v \rangle_{L^2_{\rm per}} \\
\langle u_0, L_+ u \rangle_{L^2_{\rm per}} \end{array} \right]
+ \left[ \begin{array}{c}
\langle u_0', (2 u_0 u + u^2 + v^2) v \rangle_{L^2_{\rm per}} \\
\langle u_0, (3 u_0 u + u^2 + v^2) u + u_0 v^2 \rangle_{L^2_{\rm per}}
\end{array} \right],
\end{equation}
where
\begin{equation}
\label{BBdef}
B = \left[ \begin{array}{cc} -\| u_0' \|^2_{L^2_{\rm per}} & 0 \\ 0 &
    \| u_0 \|^2_{L^2_{\rm per}} \end{array} \right] +
\left[ \begin{array}{cc}
    -\langle u_0', u_x \rangle_{L^2_{\rm per}} & \langle u_0', v
\rangle_{L^2_{\rm per}} \\
    \langle u_0, v_x \rangle_{L^2_{\rm per}} & \langle u_0, u
    \rangle_{L^2_{\rm per}} \end{array} \right].
\end{equation}
Since $\|u(\cdot,t) + i v(\cdot,t)\|_{H^2_{\rm per}} \le C_0 d(t) \leq
C_0 \epsilon$ for all $t \in \R$, the second term in the right-hand
side of (\ref{BBdef}) is of size $\mathcal{O}(\epsilon)$, hence the
matrix $B$ is invertible if $\epsilon$ is small enough. Inverting $B$
in (\ref{Bsys}), we obtain a formula for the derivatives
$\dot{\xi},\dot{\theta}$ where the right-hand side is a continuous
function of time under the mere assumption that $\psi \in
C(\R,H^2_{\rm per}(0,T))$. By a classical density argument, we
conclude that $\xi,\theta$ are differentiable in the general case, and
that their derivatives are given by (\ref{Bsys}).  Finally, the first
term in the right-hand side of (\ref{Bsys}) is of size
$\mathcal{O}(\epsilon)$, whereas the second term is
$\mathcal{O}(\epsilon^2)$, hence $|\dot{\xi}(t)| + |\dot{\theta}(t)|
\le C\epsilon$ for all $t \in \R$, where the positive constant $C$ is
independent of $t$.
\end{proof}

\appendix

\section{Explicit expressions involving Jacobi elliptic functions}

In this appendix, we derive explicit formulas the generalized
eigenvectors of the linearized operators in 
(\ref{lin-operators-flow}) by using Jacobi elliptic functions. In particular, we show how to
compute the explicit expression (\ref{expression-mu}).

Fix $\mathcal{E} \in (0,1)$ and let $k \in (0,1)$ be given by
(\ref{exact-c-plus-minus}). The periodic wave profile $u_0$ defined in (\ref{sn-periodic})
can be rewritten in the explicit form
$$
u_0(x) = \sqrt{\frac{2k^2}{1+k^2}}~{\rm sn}\Bigl(\frac{x}{\sqrt{1+k^2}}\,,\,k\Bigr) = \sqrt{\frac{2k^2}{1+k^2}}~\jmath
  \Bigl(\frac{x}{\sqrt{1+k^2}}\Bigr), \quad x \in \mathbb{R},
$$
where $\jmath(\xi) = {\rm sn}(\xi,k)$ denotes the Jacobi elliptic
function. To simplify the calculations below, it is convenient to use the
space variable $\xi = x/\sqrt{1+k^2}$ instead of $x$.

Let us recall a few properties of the Jacobi elliptic functions
${\rm sn}(\xi,k)$, ${\rm cn}(\xi,k)$, and ${\rm dn}(\xi,k)$ \cite{Lawden}.
The functions ${\rm sn}(\xi,k)$ and ${\rm cn}(\xi,k)$ are periodic with period $T = 4K(k)$,
where $K(k)$ denotes the complete elliptic integral of the first kind.
On the other hand, the function ${\rm dn}(\xi,k) = \sqrt{1 - k^2{\rm sn}
(\xi,k)^2}$ is periodic with period $2K(k)$.

We have the following expressions for the first-order derivatives of the Jacobin elliptic functions:
\begin{equation}\label{Jacobiprime}
\frac{d}{d \xi} \left[\begin{array}{c} {\rm sn}(\xi,k)\\{\rm cn}(\xi,k)
\\{\rm dn}(\xi,k) \end{array} \right] = \left[ \begin{array}{c}
{\rm cn}(\xi,k) {\rm dn}(\xi,k), \\ -{\rm sn}(\xi,k) {\rm dn}(\xi,k),
\\ -k^2 {\rm sn}(\xi,k) {\rm cn}(\xi,k).\end{array} \right]
\end{equation}
In particular, the function $\jmath(\xi) = {\rm sn}(\xi,k)$ satisfies
the differential equation
\begin{equation}\label{jode}
\frac{d^2 \jmath}{d \xi^2} = - (1+k^2) \jmath + 2 k^2 \jmath^3.
\end{equation}
Let us also introduce the incomplete elliptic integral of the second kind
\begin{equation}\label{Ellipticdef}
E(\xi,k) = \int_0^{\xi} {\rm dn}^2(y,k) dy, \quad \xi \in \mathbb{R}.
\end{equation}
This function is not periodic and we have the relation
$$
E(\xi + 2K(k),k) = E(\xi,k) + 2 E(k) \quad \mbox{\rm for all} \quad \xi \in \mathbb{R},
$$
where $E(k) := \frac12 E(2K(k),k)$ is the complete
elliptic integral of the second kind. This means that the function
$\xi \mapsto E(\xi,k)$ is linearly growing at infinity with
asymptotic rate $E(k)/K(k)$.

Using the chain rule for the operator $L_- = -\partial_x^2 + u_0^2(x) - 1$,
we obtain $\mathcal{L}_- = (1+k^2) L_-$, where
$$
\mathcal{L}_- = -\partial_{\xi}^2 - (1 + k^2) + 2 k^2 \jmath(\xi)^2.
$$
Recall that $\mathcal{L}_- \jmath = 0$.
Using the relations (\ref{Jacobiprime})--(\ref{Ellipticdef}), it is
easy to verify that
\begin{align*}
& \mathcal{L}_- \Bigl({\rm cn}(\xi,k) {\rm dn}(\xi,k)\Bigr) \,=\, -4k^2\,
{\rm cn}(\xi,k) \, {\rm dn}(\xi,k) \, {\rm sn}^2(\xi,k), \\
& \mathcal{L}_- \Bigl({\rm sn}(\xi,k) E(\xi,k)\Bigr) \,=\, - 2 \,{\rm cn}(\xi,k) \,
{\rm dn}(\xi,k) \left(1 - 2 k^2 {\rm sn}^2(\xi,k) \right), \\
& \mathcal{L}_- \Bigl(\xi {\rm sn}(\xi,k)\Bigr) \,=\, -2 \,{\rm cn}(\xi,k) \,
{\rm dn}(\xi,k).
\end{align*}
Therefore, the function
\begin{equation}
\label{function-V}
V(\xi) := {\rm cn}(\xi,k)\,{\rm dn}(\xi,k) + {\rm sn}(\xi,k)
\left[ E(\xi,k) - \frac{E(k)}{K(k)} \xi \right],
\end{equation}
is periodic with period $T = 4K(k)$ and satisfies the inhomogeneous equation
\begin{equation}\label{L0ident}
\mathcal{L}_- V = -2 \left( 1 - \frac{E(k)}{K(k)} \right) {\rm cn}(\xi,k)
\,{\rm dn}(\xi,k) = -2 \left( 1 -  \frac{E(k)}{K(k)} \right) \jmath'.
\end{equation}
Note that the numerical coefficient in (\ref{L0ident}) is nonzero because
$K(k) > E(k)$ for all $k \in (0,1)$.

Using the chain rule for the operator $M_- \,=\, \partial_x^4 - 3 \partial_x u_0^2 \partial_x + u_0^2 - 1$,
we obtain $\mathcal{M}_- = (1+k^2)^2 M_-$, where
$$
\mathcal{M}_- = \partial_{\xi}^4 - 6 k^2 \partial_{\xi} \jmath(\xi)^2
  \partial_{\xi} + 2 k^2 (1 +k^2) \jmath(\xi)^2 - (1+k^2)^2.
$$
A long but direct calculation using (\ref{Jacobiprime})
shows that the same function $V$ in (\ref{function-V}) also satisfies
\begin{equation}\label{M0ident}
\mathcal{M}_- V = 4 \left[ k^2 - \left( 1 -  \frac{E(k)}{K(k)} \right)
(1 + k^2) \right] \jmath'.
\end{equation}
Recall that $K_-(c) = M_- - c L_-$. Combining (\ref{L0ident}) and (\ref{M0ident})
and using the chain rule, we obtain
\begin{equation}\label{M0L0ident}
\left( \mathcal{M}_- - c(1+k^2) \mathcal{L}_- \right) V = \left[ 4 k^2 + 2 (c-2)
  (1+k^2) \left( 1 -  \frac{E(k)}{K(k)} \right) \right] \jmath'.
\end{equation}
Note that $\jmath$ and $V$ are orthogonal with respect to the scalar
product $\langle \cdot, \cdot \rangle_{L^2_{\rm per}}$
in $L^2_{\rm per}(-2K(k),2K(k))$ because $V$ is even and $\jmath$ is odd.

\begin{remark}
\label{remark-appendix}
The fact that both quantities $\mathcal{L}_- V$ and $\mathcal{M}_- V$ are proportional
to the same function $\jmath'$ is not an accident. Associated with the neutral mode
$(u_0',0)$, we have $L_- v = u_0'$ arising in the solutions of
the linearized evolution operator at $u_0$:
\begin{equation*}
\left[ \begin{array}{cc} 0 & L_- \\ -L_+ & 0 \end{array} \right]
\left[ \begin{array}{cc} u_0' \\ 0 \end{array} \right]\,=\,
\left[ \begin{array}{cc} 0 \\ 0 \end{array} \right] \quad \hbox{and}\quad
\left[ \begin{array}{cc} 0 & L_- \\ -L_+ & 0\end{array} \right]
\left[ \begin{array}{cc} 0 \\ v \end{array} \right]\,=\,
\left[ \begin{array}{cc} u_0' \\ 0 \end{array} \right],
\end{equation*}
hence $(0,v)$ is the generalized neutral mode. Now the higher-order operators $M_\pm$
are associated with the linearization of another flow in the hierarchy
of the integrable NLS equation, which commutes with the original flow
of (\ref{nls}), see Section \ref{sec:operator}. As
is easily verified, this implies that the same function $v$ satisfies
$M_- v = A u_0'$ for some constant $A \in \mathbb{R}$, in agreement
with (\ref{L0ident}) and (\ref{M0ident}) after the scaling transformation
from $x$ to $\xi$.
\end{remark}

We can now obtain the explicit expression (\ref{expression-mu})
from the formula (\ref{expansion-lowest-band}). Recall that
$z = \ell x$ and $U = u_0(\ell^{-1} \cdot)$. If $W = w(\ell^{-1} \cdot)$
satisfies $P_-(c,0) W = U'$, then
$$
\Bigl(P_-(c,0)W\Bigr)(z) =   (1+k^2)^{-2} \left(\mathcal{M}_- -
c(1+k^2) \mathcal{L}_-  \right) w(\xi), \quad \xi = \frac{z}{\ell\sqrt{1+k^2}}.
$$
Using the chain rule, we rewrite the formula (\ref{expansion-lowest-band}) in the equivalent form
\begin{align}
\nonumber
\mu''(0) \,=\, \frac{2\ell^2}{\|\jmath\|_{L^2_{\rm per}}^2}
\Bigl[&-4 (c-2)^2(1+k^2) \langle \jmath',(\mathcal{M}_- -
c(1+k^2) \mathcal{L}_-)^{-1} \jmath' \rangle_{L^2_{\rm per}} \\ \label{expansion-lowest-band-appendix}
&+ 3 (1+k^2)^{-1} \|\jmath'\|_{L^2_{\rm per}}^2 + (3-c)
\|\jmath\|^2_{L^2_{\rm per}}\Bigr].
\end{align}
It follows from (\ref{M0L0ident}) that
\begin{equation}
\label{expansion-app-technical}
(\mathcal{M}_- - c(1+k^2) \mathcal{L}_-)^{-1} \jmath' = \frac{V}{4 k^2 + 2 (c-2)
  (1+k^2) \left( 1 -  \frac{E(k)}{K(k)} \right)}.
\end{equation}
It remains to compute the norms and the scalar products
in the right-hand side of equation (\ref{expansion-lowest-band-appendix}).
Using the notations above, we find for all $k \in (0,1)$,
\begin{align}
\label{norm1-app}
 \|\jmath\|^2_{L^2_{\rm per}} \,&=\, \frac{4 K(k)}{k^2} \left[ 1 - \frac{E(k)}{K(k)}
 \right] > 0, \\
 \label{norm2-app}
\|\jmath'\|^2_{L^2_{\rm per}} \,&=\, \frac{4 K(k)}{3 k^2} \left[ k^2-1
+ (k^2+1) \frac{E(k)}{K(k)} \right] > 0,
\end{align}
and
\begin{equation}
\label{norm3-app}
\langle \jmath', V \rangle_{L^2_{\rm per}} =  \frac{2 K(k)}{k^2} \left[
k^2-1 + 2 \frac{E(k)}{K(k)} - \frac{E(k)^2}{K(k)^2} \right].
\end{equation}
Substituting these expressions into (\ref{expansion-lowest-band-appendix})
and finding a common denominator for all terms, we obtain the expression
(\ref{expression-mu}).

Next, using the chain rule for the operator $L_+ = -\partial_x^2 + 3 u_0^2(x) - 1$,
we obtain $\mathcal{L}_+ = (1+k^2) L_+$, where
$$
\mathcal{L}_+ = -\partial_{\xi}^2 - (1 + k^2) + 6 k^2 \jmath(\xi)^2.
$$
Recall that $\mathcal{L}_+ \jmath' = 0$.
Using the relations (\ref{Jacobiprime})--(\ref{Ellipticdef}), it is
easy to verify that
\begin{align*}
& \mathcal{L}_+ \Bigl({\rm sn}(\xi,k)\Bigr) \,=\, 4k^2\, {\rm sn}^3(\xi,k), \\
& \mathcal{L}_+ \Bigl({\rm cn}(\xi,k) {\rm dn}(\xi,k) \int_0^{\xi} \frac{{\rm sn}^2(y,k)}{{\rm dn}^2(y,k)} dy \Bigr) \,=\, -2 \,{\rm sn}(\xi,k) \,
\left(1 - 2 {\rm sn}^2(\xi,k) \right), \\
& \mathcal{L}_+ \Bigl(\xi {\rm cn}(\xi,k) {\rm dn}(\xi,k\Bigr) \,=\, 2 \,{\rm sn}(\xi,k) \,
\left(1 + k^2 - 2 k^2 {\rm sn}^2(\xi,k) \right).
\end{align*}
Therefore, the function
\begin{equation}
\label{function-U}
U(\xi) := (1-k^2) (1 + b k^2) {\rm sn}(\xi,k) - k^2 (1-k^2) {\rm cn}(\xi,k){\rm dn}(\xi,k)
\left[ \int_0^{\xi} \frac{{\rm sn}^2(y,k)}{{\rm dn}^2(y,k)} dy - b \xi \right],
\end{equation}
satisfies the inhomogeneous equation
\begin{equation}\label{LPlus}
\mathcal{L}_+ U = 2 k^2 (1-k^2) (1+b(1+k^2)) {\rm sn}(\xi,k) = 2 k^2 (1-k^2) (1+b(1+k^2))  \jmath,
\end{equation}
for an arbitrary coefficient $b \in \mathbb{R}$.

We shall find the value of $b$ from the condition that $U$ is periodic with period $T = 4K(k)$.
To do so, we recall the identity (see 16.26.6 in \cite{AS}):
$$
(1-k^2) \int_0^{\xi} \frac{dy}{{\rm dn}^2(y,k)} = E(\xi,k) - k^2 \frac{{\rm sn}(\xi,k){\rm cn}(\xi,k)}{{\rm dn}(\xi,k)}.
$$
Using this identity, we rewrite the function $U$ given by (\ref{function-U}) in the equivalent form
\begin{eqnarray}
\nonumber
U(\xi) & = & (1-k^2) (1 + b k^2) {\rm sn}(\xi,k) + k^2 {\rm sn}(\xi,k) {\rm cn}^2(\xi,k) \\
\label{function-U-equiv} & \phantom{t} & \phantom{text} -
{\rm cn}(\xi,k){\rm dn}(\xi,k) \left[ E(\xi,k) - (1-k^2) (1 + b k^2) \xi \right],
\end{eqnarray}
which is periodic if and only if $(1-k^2) (1 + b k^2) = \frac{E(k)}{ K(k)}$. Substituting this expression
into (\ref{function-U-equiv}), we finally obtain the $4K(k)$-periodic solution
\begin{eqnarray}
\label{function-U-final}
U(\xi) = \frac{E(k)}{K(k)} {\rm sn}(\xi,k) + k^2 {\rm sn}(\xi,k) {\rm cn}^2(\xi,k) -
{\rm cn}(\xi,k){\rm dn}(\xi,k) \left[ E(\xi,k) - \frac{E(k)}{K(k)} \xi \right]
\end{eqnarray}
of the inhomogeneous equation
\begin{equation}\label{LPlusUU}
\mathcal{L}_+ U = 2  \left( k^2 - 1 + (1+k^2) \frac{E(k)}{K(k)} \right)  \jmath.
\end{equation}
Note that the numerical coefficient in (\ref{LPlusUU})
is nonzero for every $k \in (0,1)$, thanks to (\ref{norm2-app}).

Using the chain rule for the operator $M_+ \,=\, \partial_x^4 - 5 \partial_x u_0^2 \partial_x - 5 u_0^4 + 15 u_0^2 - 4 + 3 \mathcal{E}^2$,
we obtain $\mathcal{M}_+ = (1+k^2)^2 M_+$, where
$$
\mathcal{M}_+ = \partial_{\xi}^4 - 10 k^2 \partial_{\xi} \jmath(\xi)^2
  \partial_{\xi} - 20 k^4  \jmath(\xi)^4 + 30 k^2 (1 +k^2) \jmath(\xi)^2 - (1+14 k^2+k^4).
$$
After a long but direct calculation, we obtain that the same function $U$ in (\ref{function-U-final}) also satisfies
\begin{equation}\label{MPlus}
\mathcal{M}_+ U = 4 \left[ 2k^4 - k^2 - 1 + (1 + 4 k^2 + k^4) \frac{E(k)}{K(k)} \right] \jmath.
\end{equation}
Combining (\ref{LPlusUU}) and (\ref{MPlus}) into $K_+(c) = M_+ - cL_+$ for $c = 2$
and using the chain rule, we obtain
\begin{equation}\label{MLPlut}
\left( \mathcal{M}_+ - 2(1+k^2) \mathcal{L}_+ \right) U = 4 k^2 \left[ k^2 - 1 + \frac{2 E(k)}{K(k)} \right] \jmath.
\end{equation}
Since $2 > 1 + k^2$, the numerical coefficient in front of $\jmath$ is positive for all $k \in (0,1)$,
thanks to (\ref{norm2-app}).

\begin{remark}
Again, we observe that both quantities $\mathcal{L}_+ U$ and $\mathcal{M}_+ U$ are proportional
to the same function $\jmath$. This is due to the generalized neutral mode $(u,0)$ associated with the neutral mode
$(0,u_0)$, which arise in the solution of $L_+ u = u_0$. See also Remark \ref{remark-appendix}.
\end{remark}

\vspace{0.5cm}

\noindent{\bf Acknowledgements.} The authors thank B. Deconinck for
pointing out to his work \cite{Decon} and for helping to compare our
analytic formula (\ref{exact-c-plus-minus}) with the results of
\cite{Decon}. The authors also thank M. Haragus for pointing to the
interwining relation (\ref{intertwining}), which helped us to extend the result
to periodic waves of large amplitudes and to prove the spectral stability
of periodic waves. D.P. is supported by the Chaire d'excellence ENSL/UJF.
He thanks members of Institut Fourier, Universit\'e Grenoble for
hospitality and support during his visit (January-June, 2014).

\end{document}